\documentclass[12pt,a4paper,reqno]{amsart}
 \usepackage{amsmath, amssymb, amsthm}
\usepackage[english]{babel} 
\usepackage{fullpage}
\usepackage{hyperref}  
\usepackage{enumerate}
\usepackage[all]{xy}

\title[Quadratic algebras and their symmetries] 
{A class of differential quadratic algebras \\ ~ \\ and their symmetries} 

\date{v1 January 2017; v2 April 2017}

\author{Giovanni Landi, Chiara Pagani}
\address[]{\textit{Giovanni Landi}  \newline \indent 
Matematica, Universit\`a di Trieste, \newline \indent
Via A. Valerio, 12/1, 34127  Trieste, Italy \newline \indent
and INFN, Trieste, Italy}
\email{landi@units.it}

\address[]{\textit{Chiara Pagani} \newline \indent  
Mathematisches Institut,  Georg-August Universit\"at G\"ottingen,
\newline \indent
Bunsenstra\ss e~3-5,~37073~G\"ottingen,~Germany}
\email{cpagani@uni-math.gwdg.de}

\numberwithin{equation}{section}

\theoremstyle{plain}
\newtheorem{thm}{Theorem}[section]
\newtheorem{lem}[thm]{Lemma}
\newtheorem{prop}[thm]{Proposition}

\newtheorem{cor}[thm]{Corollary}
\theoremstyle{definition}
\newtheorem{defi}[thm]{Definition}

\newtheorem{rem}[thm]{Remark}

\newcommand{\nn}{\nonumber}
\newcommand{\ot}{\otimes}
\newcommand{\beq}{\begin{equation}}
\newcommand{\eeq}{\end{equation}}
 \newcommand{\id}{\mathrm{id}}
\newcommand{\IC}{\mathbb{C}}
\newcommand{\IN}{\mathbb{N}}
\newcommand{\IR}{\mathbb{R}}
\newcommand{\IT}{\mathbb{T}}
\newcommand{\IS}{\mathbb{S}}
\renewcommand{\k}{\mathbb{K}} 

\newcommand{\II}{\mbox{\rm 1\hspace {-.6em} l}}
\newcommand{\cR}{\mathcal{R}}

\newcommand{\ii}{\mathrm{i}}
\newcommand{\dd}{\mathrm{d}}

\newcommand{\dx}{\dd x}
\newcommand{\alg}{\mathrm{\bold{A}}}
\newcommand{\cc}{\mathrm{c}} 

\newcommand{\alp}{\alg_{\ell,p}}
\newcommand{\omlp}{\Omega_{\ell,p}}

\newcommand{\mlp}{\mathrm{\bold{M}}_{\ell,p}}
\newcommand{\A}{M}

\newcommand{\arlp}{\alg({\IR}^4_{\ell,p})}
\newcommand{\rlp}{{\IR}^4_{\ell,p}}
\newcommand{\slp}{{\IS}^3_{\ell,p}}
\newcommand{\aslp}{\alg({\IS}^3_{\ell,p})}
\newcommand{\olp}{O_{\ell,p}(4)}

\newcommand{\Au}{\alg_\textbf{u}}
\newcommand{\Su}{{\IS}^3_\textbf{u}}

\renewcommand{\int}{\{0,1,2,3\}}

\newcommand\nt {{\tilde{\nu}}}
\newcommand\mn {{\mu_{\nu}}}
\newcommand\mt{{\tilde{\mu}_{\nu}}}

\begin{document}


\begin{abstract}
We study a multi-parametric family of quadratic algebras in four generators, which includes coordinate algebras of noncommutative four-planes and, as quotient algebras, noncommutative three-spheres. 
Particular subfamilies comprise Sklyanin algebras and Connes--Dubois-Violette planes.
We determine quantum groups of symmetries for the general algebras and construct finite-dimensional covariant differential calculi.
\end{abstract}

\maketitle

\tableofcontents
\parskip 1.25ex

\newpage
\section{Introduction}

Many important examples of noncommutative spaces and quantum groups (from FRT bialgebras  and Woronowicz quantum groups, to Manin's quantum plane, $\theta$-planes and spheres, and beyond ...) have a description as quadratic algebras, finitely generated and finitely presented, or as quotients of quadratic algebras. 
In this paper we introduce a multi-parametric family of non commutative quadratic algebras $\alp$ (over a ground field $\k$), depending on parameters $\ell = (\ell_{\mu\nu})$ and 
$p=(p_{\mu\nu})$, with ${\mu,\nu}=0,1,2,3$, that obey some minimal conditions (see Definition \ref{def:lp}). 
The algebras $\alp$ are generated by degree-one elements $x_\mu$, $\mu=0,1,2,3$ with defining relations 
in degree two given by
$$
x_\mu x_\nu = \ell_{\mu \nu} x_\nu x_\mu + p_{\mu\nu} x_{\nu'} x_{\mu'} \quad \forall \mu,\nu \in \int
$$ 
(see below for the  notation used). The family of algebras $\alp$ have well-known sub-families: with special classes of the parameters  we can recover relevant algebras, notably Sklyanin algebras, or (the algebras of) $\theta$-planes and Connes--Dubois-Violette four-planes.

When all  parameters $p_{\mu \nu}$ vanish and all $\ell_{\mu \nu}$ are equal to $1$ the algebra $\alp$ becomes the commutative algebra of polynomials in four coordinates $x_0, \dots,x_3$ and we recover the classical case of a commutative four-plane. Thus we may think of $\alp$ as being 
$\alp =: \arlp$, that is the coordinate algebra of a noncommutative four-plane $\rlp$. 

The algebras $\alp$ can be given a finite-dimensional differential calculus. In \S \ref{sec:calculus} we construct a differential graded algebra $\omlp$, with $\alp$ as degree-zero part, together with a degree one operator $\dd$ (the differential) obeying the Leibniz rule. We show that  the calculus is finite of order four. One of its peculiarity is that 
in the top component $\omlp^4$, in addition to `usual' forms $\dx_\nu \dx_\mu \dx_\tau \dx_\sigma$ 
with indices all different, 
there are also `quantum' elements of the type
$\dx_\nu \dx_\mu \dx_\nu \dx_\mu$,  with $\nu \neq \mu$. Nevertheless, the space of 
four-forms turns out to be one-dimensional, with a volume form $\omega$ which we explicitly determine.

The center of the algebra $\alp$ is in general difficult to determine completely (and could be rather big for some parameters). We single out (in Proposition \ref{prop:R}) a condition on the parameters $\ell$ and $p$ under which certain degree-two elements of the type
$R_c:=\sum_{\mu=0}^3 \cc_\mu x_\mu^2$, 
belong to the center of  the algebra  $\alp$, for polynomial coefficients  $c_\mu \in \k[\ell_{\mu \nu}, p_{\mu \nu}]$.  One possibility is that all these coefficients are $c_\mu=1$ for all indices (in Corollary \ref{cor-central}) and using the corresponding central element $R:=\sum_{\mu=0}^3 x_\mu^2$ we can introduce a family of noncommutative algebras, $\aslp:= \alp / \langle \sum_{\mu=0}^3  x_\mu^2-1 \rangle$  describing the algebra of coordinate functions of quantum three-spheres $\slp$. The calculus $(\omlp, \dd)$ descends to a differential calculus on $\slp$.

In the last part of the paper we focus on the study of quantum groups of symmetries for the algebras $\alp$. We construct a bialgebra $\mlp$ with a coaction $\delta: \alp \to \mlp \ot \alp $ which endows $\alp$ with the structure of a left $\mlp$-comodule algebra. The bialgebra 
$\mlp$ is a quantum matrix algebra defined by quadratic relations among its noncommutative coordinate functions. 
The coaction $\delta$ is required to be compatible with the differential $\dd$, that is it extends to a coaction on the graded algebra $\omlp$. Thus, the differential calculus $(\omlp, \dd)$ on $\alp$ is covariant. In the classical limit 
(with all parameters $p_{\mu \nu}$ vanishing and all $\ell_{\mu \nu}= 1$), the bialgebra $\mlp$ reduces, as expected, to the commutative coordinate bialgebra of $4 \times 4$ matrices in ${\rm Mat}_4(\k)$. 
Future work will be devoted to the study of quotient algebras of the bialgebra $\mlp$ describing matrix quantum groups, and in particular a quantum group of orthogonal matrices acting on $\slp$.

\subsubsection*{Acknowledgments.} We thank Michel Dubois-Violette for many useful discussions and suggestions and Alessandro Logar for his extensive help with symbolic computations. Paul Smith made useful remarks via  email.

\section{The quadratic algebras $\alp$}\label{sec:alp}

We work over a field $\k$ of characteristic zero and denote  by $1$ its (multiplicative) unit.
Greek indices will run in $\{0,1,2,3\}$; latin indices run in $\{1,2,3\}$.
Having fixed two distinct indices $\mu, \nu \in \int$ with (say) $\mu>\nu$, we denote by $\mu',\nu'$ the (uniquely defined) indices $\mu',\nu' \in \int \backslash \{\mu,\nu\}$ with $\mu' >\nu'$.  When $\nu=\mu$, we define $\nu'=\mu'=\mu=\nu$.  Clearly $(\mu'',\nu'')= (\mu,\nu)$.  We sometimes write $(\mu,\nu)'$ to indicate $(\mu',\nu')$.
With this notation, in the following each identity which holds for $\mu, \nu$, will also hold by replacing $\mu \leftrightarrow \mu'$ and  
$\nu \leftrightarrow \nu'$, provided the replacement is done simultaneously.

\subsection{Generators and relations}
We study a family of quadratic algebras finitely presented in terms 
of generators and relations and depending on a set of parameters. 
\begin{defi}\label{def:lp}
For all $\mu,\nu \in \int$ let $\ell_{\mu \nu}$ and $p_{\mu \nu} \in \k$ satisfy the conditions
\begin{enumerate}[(a)]
\item\label{l} \quad
$\ell_{\mu \mu} =1$ ; \quad 
$\ell_{\mu \nu}= \ell_{\nu \mu}$ ; \quad $\ell_{\mu' \nu'}=\ell_{\mu \nu}$ ; 
\item \label{p} \quad
$p_{\mu \nu}= -p_{\nu \mu}$ ;
\item \label{lp} \quad
$\ell_{\mu \nu}^2+ p_{\mu \nu}  p_{\nu' \mu'}  =1 $. 
\end{enumerate}
We denote by $\alp$ the graded associative $\k$-algebra (with $\k$ as degree zero component) generated in degree one by algebra generators $x_\mu$, $\mu \in \int$ and defining relations 
in degree two given by
\beq\label{comm-rel-x}
x_\mu x_\nu = \ell_{\mu \nu} x_\nu x_\mu + p_{\mu\nu} x_{\nu'} x_{\mu'} \quad \forall \mu,\nu \in \int.
\eeq
\end{defi}
 
\noindent
Due to the above conditions on the parameters $\ell$'s and $p$'s, one easily verifies 
that there are no additional relations in degree two: when using the relations \eqref{comm-rel-x} 
in the right hand side (for the proper pairs of indices) one obtains an identity,
\begin{align*}
x_\mu x_\nu & = \ell_{\mu \nu} x_\nu x_\mu + p_{\mu\nu} x_{\nu'} x_{\mu'} =
\ell_{\mu \nu} (\ell_{\nu \mu} x_\mu x_\nu + p_{\nu\mu} x_{\mu'} x_{\nu'} ) + p_{\mu\nu}
(\ell_{\nu' \mu'} x_{\mu'} x_{\nu'} + p_{\nu'\mu'} x_{\mu} x_{\nu} )
\\
& = (\ell_{\mu \nu}^2 + p_{\mu\nu}p_{\nu'\mu'})x_\mu x_\nu +
\ell_{\mu \nu}(p_{\nu\mu}  + p_{\mu\nu}) x_{\mu'} x_{\nu'} = x_\mu x_\nu .
\end{align*}

\noindent
We stress that indeed it is enough to consider equations \eqref{comm-rel-x} for $\mu>\nu$; those for $\mu<\nu$  are then implied. Indeed, assume \eqref{comm-rel-x} holds for indices $\mu>\nu$ fixed, then
\begin{align*}
\ell_{\nu \mu} x_\mu x_\nu + p_{\nu\mu} x_{\mu'} x_{\nu'} & = 
\ell_{\nu \mu} (\ell_{\mu \nu} x_\nu x_\mu + p_{\mu\nu} x_{\nu'} x_{\mu'}) 
+ p_{\nu\mu} (\ell_{\mu' \nu'} x_{\nu'} x_{\mu'} + p_{\mu'\nu'} x_{\nu} x_{\mu}) 
\\
& =   (\ell_{\mu \nu}^2 + p_{\nu\mu} p_{\mu'\nu'})     x_\nu x_\mu
 + (\ell_{\nu \mu} p_{\mu\nu} 
+ p_{\nu\mu} \ell_{\mu \nu}) x_{\nu'} x_{\mu'}   =x_\nu x_\mu \; .
\end{align*}
Generically, for the family of quadratic algebras $\alp$ the number of independent parameters is six. There are six parameters 
$\ell_{j0}$ and $p_{j0}$, with $j=1,2,3$, and three parameters $p_{jk}$, with $j > k\in\{1,2,3\}$. 
They are related by the three conditions \eqref{lp} of Definition~\ref{def:lp}.

Explicitly the relevant commutation relations are
\begin{align} \label{exre}
x_1 x_0 = \ell_{1 0} x_0 x_1 + p_{10} x_2 x_3 \, , \qquad 
x_3 x_2 = \ell_{10} x_2 x_3 + p_{32} x_{0} x_{1} \, , \nn
\\
x_2 x_0 = \ell_{2 0} x_0 x_2 + p_{20} x_{1} x_3 \, , \qquad
x_3 x_1 = \ell_{20} x_1 x_3 + p_{31} x_{0} x_2 \, ,
\\
x_3 x_0 = \ell_{3 0} x_0 x_3 + p_{30} x_{1} x_2 \, , \qquad
x_2 x_1 = \ell_{30} x_1 x_2 + p_{21} x_{0} x_3 \, . \nn
\end{align}
One important aspect of the relations \eqref{comm-rel-x}, as  
 shown by their explicit form 
\eqref{exre} is that they are ordered so that the six ordered binomials $x_\mu x_\nu$, with $\mu < \nu$, together with the four binomials $x_\mu^2$, form a basis of degree-two polynomials. This fact will turn out to be useful later on.

\begin{rem}
The quadratic relations \eqref{comm-rel-x} that define the algebras $\alp$ can be expressed in the form
\beq\label{rmrel}
x_\mu x_\nu = \sum_{\sigma, \tau} \cR_{\mu \nu}^{~ ~\sigma \tau} x_\sigma x_\tau \, ,
\qquad \cR_{\mu \nu}^{~ ~\sigma \tau}:= \ell_{\mu \nu} \delta_{\mu \tau} \delta_{\nu \sigma} + p_{\mu \nu} \delta_{\mu' \tau} \delta_{\nu' \sigma} 
\eeq
for all pair of indices $(\mu, \nu)$. 
It is easy to see that the matrix $\cR \in {\rm Mat}_{4^2}(\k)$ is invertible and involutive, that is 
$\cR^2= \II \otimes \II$, for $\II$ the identity matrix; indeed one easily finds:
$$
\sum_{\alpha, \beta} \cR_{\mu \nu}^{~ ~\alpha \beta} \cR_{\alpha \beta}^{~ ~\sigma \tau}
= (\ell_{\mu \nu}^2 + p_{\mu \nu} p_{\nu' \mu'} )\delta_{\mu \sigma} \delta_{\nu \tau}  
=\delta_{\mu \sigma} \delta_{\nu \tau},
$$
using condition \eqref{lp} in the Definition~\ref{def:lp}. 
On the other hand, for generic parameters $\ell_{\mu \nu}$ and $p_{\mu \nu}$, the matrix $\cR$ does not satisfy the quantum Yang--Baxter equations. An analysis of these equations for our algebras $\alp$ will be reported elsewhere. 
\end{rem}

\subsection{Examples}\label{ex:class}
The family of algebras $\alp$ comprises a few interesting subfamilies.

\subsubsection{Extreme cases}
There are some `extreme' families.

\noindent
$\bullet$ \quad 
All $p_{\mu \nu}=0$. Conditions \eqref{lp} of Definition~\ref{def:lp} reduce to 
$\ell_{\mu \nu} = \pm 1$ and we have:
$$
x_\mu x_\nu = \pm x_\nu x_\mu .
$$
In particular, if all $\ell_{\mu \nu} = 1$, the algebra $\alp$ is the commutative 
$\k$-algebra in four-generators. 
We stress that $\ell_{\mu \nu} = \pm 1$ does not imply $p_{\mu \nu}=0$
but only that $p_{\mu \nu}  p_{\nu' \mu'}=0$.

\noindent  
$\bullet$ \quad
All $\ell_{\mu \nu}=1$, but $p_{\mu \nu}$ not all zero. As mentioned, condition $\ell_{\mu \nu}= 1$ for all $\mu,\nu$ does not force the vanishing of all $p_{\mu \nu}$. For conditions \eqref{lp} in Definition~\ref{def:lp} to be satisfied, it is enough that  $p_{\mu \nu}  p_{\nu' \mu'} =0$ for each pair of indices $\mu,\nu$. 
We shall mention examples of these occurrences in the next section \S \ref{sec:skly}
and later on in \S\ref{sec:cdv}.

\noindent
$\bullet$ \quad
All $\ell_{\mu \nu}=0$. For conditions \eqref{lp} of Definition~\ref{def:lp} one needs 
$p_{\mu \nu}$ and $p_{\nu' \mu'}$ different from zero with $p_{\nu' \mu'} = (p_{\mu \nu})^{-1}$. 
The relations \eqref{exre} for the corresponding algebra $\alp$ become:
\begin{align} \label{exrebis}
x_1 x_0 = p_{10} \, x_2 x_3 \, , \qquad 
x_3 x_2 = - (p_{10})^{-1} \, x_{0} x_{1} \, , \nn
\\
x_2 x_0 = p_{20} x_{1} x_3 \, , \qquad
x_3 x_1 =  - (p_{20})^{-1} \, x_{0} x_2 \, ,
\\
x_3 x_0 = p_{30} x_{1} x_2 \, , \qquad
x_2 x_1 = - (p_{30})^{-1} \, x_{0} x_3 \, . \nn
\end{align}

\subsubsection{Sklyanin algebras}\label{sec:skly}
Another important subfamily is made by the Sklyanin algebras that we shall now briefly describe.

Let $\mathrm{a}, \mathrm{b}, \mathrm{c} \in \k$ with $\mathrm{a},\mathrm{c} \neq 1$ 
and $\mathrm{b} \neq -1$. Define parameters 
\begin{align}\label{lp-skly}
\ell_{01} & :=\frac{1 + \mathrm{a}}{1-\mathrm{a}} \, , \quad 
p_{01}:= \frac{2 \mathrm{a}}{1-\mathrm{a}}= (\ell_{01}-1) \, , \quad
p_{23}:= \frac{2}{1-\mathrm{a}}=   (1+\ell_{01} ) 
\nn \\
\ell_{02} & :=\frac{1 - \mathrm{b}}{1+\mathrm{b}} \, , \quad
p_{02}:= \frac{2\mathrm{b}}{1+\mathrm{b}} = (1-\ell_{02}) \, , \quad
p_{13}:= \frac{-2}{1+\mathrm{b}}= -( 1+\ell_{02}) 
\nn \\
\ell_{03} & :=\frac{1 + \mathrm{c}}{1-\mathrm{c}} \, , \quad
p_{03}:= \frac{2 \mathrm{c}}{1-\mathrm{c}} = ( \ell_{03}-1) \, , \quad
p_{12}:= \frac{2}{1-\mathrm{c}}=(  1+\ell_{03}) \, ,
\end{align}
with $\ell_{\mu \nu}= \ell_{\nu \mu}=\ell_{\mu' \nu'}$,  $\ell_{\mu \mu} =1 $ and
$p_{\mu \nu}= -p_{\nu \mu}$.
Then it is easy to show that for all $\mu, \nu$, it holds that 
$\ell_{\mu \nu}^2+ p_{\mu \nu}  p_{\nu' \mu'}  =1$, that is condition \eqref{lp} in 
Definition~\ref{def:lp}.

The  family of algebras $\alp$ corresponding to this choice of the parameters  were introduced by Sklyanin in \cite{Skly} in the context of quantum Yang-Baxter equations and extensively studied in \cite{smith}.
In fact for a  proper Sklyanin algebra one needs the additional condition 
\beq\label{cond}
\mathrm{a} + \mathrm{b} + \mathrm{c} + \mathrm{a} \mathrm{b} \mathrm{c}=0 .
\eeq
This, or equivalently $(1+ \mathrm{a})(1+\mathrm{b})(1+\mathrm{c})
=(1- \mathrm{a})(1-\mathrm{b})(1-\mathrm{c})$, reads
\beq\label{cond-l}
\ell_{02}=\ell_{01}\ell_{03} , 
\eeq
thus giving an additional constraint on the $\ell$'s.
Originally, the algebra was introduced as the  quadratic algebra generated by degree one elements $\mathrm{x}_\mu$ with relations
\begin{align}\label{skly-k}
[\mathrm{x}_0, \mathrm{x}_1]_{-} =  \mathrm{a} [\mathrm{x}_2,\mathrm{x}_3]_{+} \, , \qquad &    [\mathrm{x}_2,\mathrm{x}_3]_{-}=[\mathrm{x}_0, \mathrm{x}_1]_{+}  \, ,
\nn \\
[\mathrm{x}_0, \mathrm{x}_2]_{-} =   \mathrm{b} [\mathrm{x}_3,\mathrm{x}_1]_{+}  \, , \qquad &   [\mathrm{x}_3,\mathrm{x}_1]_{-}= [\mathrm{x}_0, \mathrm{x}_2]_{+}  \, , 
\nn \\
[\mathrm{x}_0, \mathrm{x}_3]_{-} =   \mathrm{c} [\mathrm{x}_1,\mathrm{x}_2]_{+}  \, , \qquad &   [\mathrm{x}_1,\mathrm{x}_2]_{-} =[\mathrm{x}_0, \mathrm{x}_3]_{+} \, .
\end{align}
where $[~,~]_-$ and $[~,~]_+$ stand for the commutator and anticommutator respectively. 

\begin{rem}
When $\mathrm{a}=\mathrm{b}=\mathrm{c}=0$, one has $\ell_{01}=\ell_{02}=\ell_{03}=1$ and
$p_{01}=p_{02}=p_{03}=0$ while $p_{23}=p_{31}=p_{12}=2$. In the corresponding algebra the generator $\mathrm{x}_0$ is central: $[\mathrm{x}_0, \mathrm{x}_1]_{-}
=[\mathrm{x}_0, \mathrm{x}_2]_{-}=[\mathrm{x}_0, \mathrm{x}_3]_{-}=0, $ and the defining relations reduce to 
\begin{align}\label{skly-k00}
[\mathrm{x}_2,\mathrm{x}_3]_{-}= 2 \, \mathrm{x}_0 \mathrm{x}_1   \, , \qquad 
[\mathrm{x}_1,\mathrm{x}_3]_{-}= 2 \, \mathrm{x}_0 \mathrm{x}_2  \, , \qquad
[\mathrm{x}_1,\mathrm{x}_2]_{-} = 2 \, \mathrm{x}_0 \mathrm{x}_3 \, .
\end{align}
\end{rem}

\subsection{$*$-structures}\label{sec:*lp}
Let $\k=\IC$. The algebra $\alp$ is made into a $*$-algebra by taking the generators to be hermitian ones:  
\beq\label{*x}
*(x_\mu)=x_\mu \, .
\eeq
This requires that the deformation parameters obey the conditions
\beq\label{*lp}
\bar{\ell}_{\mu \nu}={\ell}_{\mu \nu} \, , \quad \mbox{and} 
\quad \overline{p}_{\mu \nu}= p_{\nu \mu} \, .
\eeq
Again we have important subfamilies that we describe next. 

\subsubsection{Connes--Dubois-Violette four-planes $\IR^4_\mathbf{u}$ }\label{sec:cdv}
Let $\k=\IC$.  In \cite[\S 2]{cdv}  the authors introduce a three-parameter family $\IR^4_\mathbf{u}$ of deformations of the
four-dimensional Euclidean space $\IR^4$ by solving some $K$-theoretic equations put forward in \cite{cl}.
The noncommutative four-plane $\IR^4_\mathbf{u}$ is the quantum space  dual  to an algebra $\Au=\IC_{alg}(\IR^4_\mathbf{u})$, with parameters $\mathbf{u}=(e^{i\varphi_1},e^{i\varphi_2},e^{i\varphi_3}) \in \IT^3$. The unital $*$-algebra  $\Au$  is generated by elements $z_\mu$, $z_\mu^*=*(z_\mu)$, $\mu=0,1,2,3$ satisfying the relations 
\begin{align}\label{rel-z}
&z_k z_0^*-z_0 z_k^* + \sum_{\mu, \nu}\epsilon_{klm} z_l z_m^* = 0 \, ,\qquad   
\nn
\\
&z_0^* z_k - z_k^* z_0 + \sum_{\mu, \nu} \epsilon_{klm} z_l^* z_m =0 \, ,\qquad  \forall k=1,2,3 ,
\end{align}
where $\epsilon_{klm}$ is completely antisymmetric in $k,l,m \in \{1,2,3\}$, with $\epsilon_{123}=1$, and
$$
\sum_{\mu=0}^3 \left( z_\mu z_\mu^* - z_\mu^* z_\mu \right) =0 .
$$ 
Moreover, the generators satisfy  
\beq\label{*z}
z_\mu^* = \sum_\nu \Lambda_{\mu \nu} z_\nu 
\eeq
 where $ \Lambda_{\mu \nu}$ are the entries of a symmetric unitary matrix 
\beq\label{mLambda}
\Lambda=\begin{pmatrix}
e^{-2i\varphi_0} & & &
\\
& e^{-2i\varphi_1} & &
\\
& & e^{-2i\varphi_2} &
\\ & & & e^{-2i\varphi_3}
\end{pmatrix}\; .
\eeq
Due to an overall symmetry, $z_\mu \mapsto \sum_\nu 
\rho S_{\mu \nu} z_\nu$, for $\rho \in U(1)$ and $S \in SO(4)$, one can further assume  one of the angles $\varphi_\mu$ to vanish, say $\varphi_0=0$, thus
$\Au$ with $\mathbf{u} \in \IT^3$.
Finally, by rescaling the generators $z_\mu$ the algebra $\Au=\IC_{alg}(\IR^4_\mathbf{u})$ admits hermitian generators as
$$
x_\mu:=e^{-i\varphi_\mu} z_\mu , \qquad *(x_\mu)=x_\mu^*=x_\mu .
$$ 

The algebras $\Au$ belong to the family of  quadratic algebras $\alp$ introduced in Definition \ref{def:lp} above, as we will now show. 

Recall that for fixed two indices $\mu, \nu \in \int$ with $\mu>\nu$, we denote by $\mu',\nu'$ the 
unique indices $\mu',\nu' \in \int \backslash \{\mu,\nu\}$ with $\mu' >\nu'$; one has $(-1)^{\mu' + \nu'} =(-1)^{\mu+ \nu}$. 
Let $\lambda_\mu:= e^{-2i\varphi_\mu}$ be the entries of the matrix $\Lambda$ in \eqref{mLambda}. 
For $\mu, \nu \in \int$ consider then 
\beq\label{abc}
a_{\mu \nu}:=\lambda_\mu \lambda_{\nu'} + \lambda_{\nu}\lambda_{\mu'}
\, , \quad
b_{\mu \nu}:=\lambda_\mu \lambda_{\mu'} + \lambda_{\nu}\lambda_{\nu'}
\, , \quad
c_{\mu \nu}:= (-1)^{\mu + \nu} (\lambda_{\mu'}^2 - \lambda_{\nu'}^2) \, .
\eeq
Clearly $c_{\mu \nu}$ is antisymmetric, while both $a_{\mu \nu}$ and $b_{\mu \nu}$ are symmetric with in addition 
$a_{\mu \nu}=a_{\mu' \nu'}$ and $b_{\mu \nu}=b_{\mu' \nu'}$. Some little algebra also shows that 
\begin{enumerate}[(i)]
\item\label{iv} \quad
$a_{\mu \nu}^2=b_{\mu \nu}^2 + c_{\mu \nu}c_{\nu' \mu'}$ ,
\item\label{v} \quad
$\lambda_{\nu} a_{\mu \nu} -\lambda_\mu b_{\mu \nu}
= (-1)^{\nu + \mu} \lambda_{\mu'} c_{\nu' \mu'}$ .
\end{enumerate}

\begin{prop}
The generators $z_\mu$ of the algebra $\Au$ satisfy the following relations
\beq\label{comm-rel-az}
a_{\mu \nu} z_\mu z_\nu = b_{\mu \nu} z_\nu z_\mu + c_{\mu\nu} z_{\nu'} z_{\mu'}, \quad \forall \mu,\nu \in \int.
\eeq
In particular, for $\{\lambda_{\mu}, \mu=0,\dots, 3\}$  such that $a_{\mu \nu} \neq 0$ 
for all $\mu, \nu$, equations \eqref{comm-rel-az} hold if and only if equations \eqref{rel-z} hold,  that is the generators $z_\mu$ satisfy the commutation relations \eqref{rel-z} if and only if they satisfy relations \eqref{comm-rel-az}.
\end{prop}
\begin{proof} 
Firstly note that for all pairs of indices $\mu,\nu \in \int$, by using the notation $\mu ',\nu '$  
equations \eqref{rel-z} can be rewritten respectively as
\begin{align}
& z_\nu z^*_\mu    +(-1)^{\mu' + \nu'}  z_{\nu'} z^*_{\mu'} = z_\mu z^*_\nu + (-1)^{\mu' + \nu'} z_{\mu'} z^*_{\nu'} 
\label{eq1*}
\\
& z^*_\mu z_\nu - (-1)^{\mu' + \nu'} z^*_{\mu'} z_{\nu'} =  z^*_\nu z_\mu  -  (-1)^{\mu' + \nu'} z^*_{\nu'} z_{\mu'}
\label{eq2*}
\end{align}
or, using \eqref{*z}, as
\begin{align}
&\lambda_\mu z_\nu z_\mu - \lambda_\nu z_\mu z_\nu = (-1)^{\mu' + \nu'} \left( \lambda_{\nu'} z_{\mu'} z_{\nu'} - \lambda_{\mu'} z_{\nu'} z_{\mu'} \right)
\label{eq1bis}
\\
&\lambda_\mu z_\mu z_\nu - \lambda_\nu z_\nu z_\mu = (-1)^{\mu' + \nu'} \left( \lambda_{\mu'} z_{\mu'} z_{\nu'} - \lambda_{\nu'} z_{\nu'} z_{\mu'}\right) .
\label{eq2bis}
\end{align}
Since the $\lambda_\mu$ never vanish, one can take the difference of $\lambda_{\nu'}$ times equation \eqref{eq1bis} with $\lambda_{\mu'}$ times equation \eqref{eq2bis}, thus obtaining 
$$
(\lambda_\mu \lambda_{\nu'} + \lambda_\nu \lambda_{\mu'})z_\nu z_\mu
 - (\lambda_\nu \lambda_{\nu'} + \lambda_{\mu}\lambda_{\mu'} )z_\mu z_\nu 
= (-1)^{\mu' + \nu'} 
(\lambda_{\nu'}^2 - \lambda_{\mu'}^2) z_{\mu'} z_{\nu'} 
$$
that is $a_{\nu \mu} z_\nu z_\mu = b_{\nu \mu} z_\mu z_\nu + c_{\nu\mu} z_{\mu'} z_{\nu'}.
$

\noindent
Conversely, suppose that \eqref{comm-rel-az} hold; then, by using the symmetry properties of the parameters $a_{\mu \nu}$,  $b_{\mu \nu}$ and $c_{\mu \nu}$ given above and in particular \eqref{iv}, \eqref{v}, we obtain \eqref{eq1bis}:
\begin{align*}
\lambda_\mu a_{\nu \mu} z_\nu z_\mu & =  
\lambda_\mu 
b_{\nu \mu} z_\mu z_\nu + \lambda_\mu  c_{\nu \mu} z_{\mu'} z_{\nu'}
\\
& =  
\left(  \lambda_\nu a_{\mu \nu} - (-1)^{\mu + \nu} \lambda_{\mu'} c_{\nu' \mu'}
 \right) z_\mu z_\nu +
 (-1)^{\mu' + \nu'}\left(
\lambda_{\nu'} a_{\mu' \nu'}  - \lambda_{\mu'} b_{\mu' \nu'} 
\right)z_{\mu'} z_{\nu'}
\\
&=  \lambda_\nu a_{\mu \nu} z_\mu z_\nu +
 (-1)^{\mu' + \nu'}
\lambda_{\nu'} a_{\mu \nu} z_{\mu'} z_{\nu'} - (-1)^{\mu + \nu} \lambda_{\mu'} ( c_{\nu' \mu'}
 z_\mu z_\nu  + b_{\mu \nu} z_{\mu'} z_{\nu'})
\\
&= \lambda_\nu a_{\mu \nu} z_\mu z_\nu +
 (-1)^{\mu' + \nu'}
\lambda_{\nu'} a_{\mu \nu} z_{\mu'} z_{\nu'} - (-1)^{\mu + \nu} \lambda_{\mu'} a_{\mu \nu} z_{\nu'} z_{\mu'})
\end{align*} 
that is, up to multiplication by $a_{\mu \nu} \neq 0$, just \eqref{eq1bis}.
With an analogous procedure, starting from $\lambda_\mu a_{\mu \nu} z_\mu z_\nu$, we obtain \eqref{eq2bis} out of \eqref{comm-rel-az}.
\end{proof} 
  
Hence we have that the defining commutation relations for the generators of $\Au$ can equivalently be  rewritten as in \eqref{comm-rel-az} (for $a_{\mu \nu} \neq 0$).
In order to describe  the algebras $\Au$ as $\alp$ algebras, we need the following additional notation.

Let $\{\lambda_{\mu}, \mu=0,\dots, 3\}$ be 
such that $a_{\mu \nu} \neq 0$ for all $\mu, \nu$. Define
\beq\label{def:l-p}
\ell_{\mu \nu}:= 
a_{\mu \nu}^{-1} b_{\mu \nu} 
\quad ; \quad
q_{\mu \nu}:= 
a_{\mu \nu}^{-1} c_{\mu \nu} 
\eeq
Then the following identities hold
\begin{enumerate}[\rm(i)]\setcounter{enumi}{2}
\item\label{a} \quad
$\ell_{\mu \nu}= \ell_{\nu \mu}=\ell_{\mu' \nu'}=\ell_{\nu' \mu'}$ ; 
\quad $\ell_{\mu \nu'}=\ell_{\mu' \nu}$ ; 
\quad $q_{\mu \nu}= -q_{\nu \mu}$ ,
\item \label{b} \quad
$\ell_{\mu \nu}^2+ q_{\mu \nu}  q_{\nu' \mu'} = 1$ ,
\item\label{c} \quad
$\lambda_{\mu} \ell_{\mu \nu} =\lambda_\nu + (-1)^{\mu + \nu} \lambda_{\mu'} q_{\mu' \nu'}$ .
\end{enumerate}
The above are easily obtained by using the symmetry properties of the parameters 
$a_{\mu \nu}$,  $b_{\mu \nu}$ and $c_{\mu \nu}$ and the relations \eqref{iv}, \eqref{v}. 
Thus for parameters $\lambda$'s for which $a_{\mu \nu} \neq 0$, the relations \eqref{comm-rel-az} are equivalent to 
\beq\label{comm-rel-z}
z_\mu z_\nu = \ell_{\mu \nu} z_\nu z_\mu + q_{\mu\nu} z_{\nu'} z_{\mu'}, \quad \forall \mu,\nu \in \int.
\eeq

If we consider the real generators $x_\mu=e^{-i\varphi_\mu} z_\mu$, 
the commutation relations \eqref{comm-rel-z} read 
\beq\label{cdv-x}
x_\mu x_\nu = \ell_{\mu \nu} x_\nu x_\mu + p_{\mu\nu} x_{\nu'} x_{\mu'}\quad \forall \mu,\nu \in \int
\eeq
where 
$$
p_{\mu\nu}  = q_{\mu\nu}  e^{i(-\varphi_\mu -\varphi_\nu + \varphi_{\mu'} +\varphi_{\nu'})}.
$$
From properties \eqref{a}-\eqref{c} above, 
the parameters $\ell_{\mu\nu}$ and $p_{\mu\nu}$ satisfy the conditions in 
Definition~\ref{def:lp}. Moreover, from $a_{\mu \nu} \bar{b}_{\mu \nu} = \bar{a}_{\mu \nu} b_{\mu \nu}$ it follows that $\bar \ell_{\mu \nu}= \ell_{\mu \nu}$. And also, being 
$a_{\mu \nu} \bar{c}_{\mu \nu} \lambda_{\mu'} \lambda_{\nu'}= \bar{a}_{\mu \nu} c_{\nu \mu}  \lambda_\mu  \lambda_\nu $, one has that 
$\bar q_{\mu \nu} \bar \lambda_\mu  \bar \lambda_\nu \lambda_{\mu'} \lambda_{\nu'} = q_{\nu \mu}$ and in turn $\bar p_{\mu \nu}=p_{\nu \mu}$. Thus conditions \eqref{*lp} are satisfied and the corresponding algebra $\alp$ is a $*$-algebra.

While equations \eqref{comm-rel-z} (or \eqref{cdv-x}) are equivalent to  \eqref{rel-z}, they
are easier to handle, at least for the purposes of the present paper. In particular, if we use the lexicographic order $x_0<x_1<x_2<x_3$ for the generators $x_\mu$ of $\Au$, we can use  \eqref{cdv-x} to rewrite elements of $\Au$ in terms of ordered monomials which are independent and thus can be compared.

\begin{rem}\label{rem:no-abc} 

Thus, for those $\{\lambda_{\mu}, \mu=0,\dots, 3\}$ (or $u \in \IT^3$ taking $\varphi_0=0$) such that  $a_{\mu \nu}$ are non zero, the relations \eqref{comm-rel-z} or \eqref{cdv-x} give a different parametrization of the $*$-algebra $\Au$ of the noncommutative four-plane $\IR^4_\mathbf{u}$. 
Let us have a closer look at the number of actual parameters $\ell$ and $p$ entering the construction for these algebras $\Au$.
Being $a_{\mu \nu}$ and $b_{\mu \nu}$ symmetric with 
$a_{\mu \nu}=a_{\mu' \nu'}$ and $b_{\mu \nu}=b_{\mu' \nu'}$ (see after \eqref{abc}), 
\textit{a priori} there are only $a_{01}$, $a_{02}$, $a_{03}$ and $b_{01}$, $b_{02}$, $b_{03}$  (say) which are distinct. 
A further direct computation shows that
\beq\label{ab}
a_{01}=a_{02} \quad ; \quad b_{01}=a_{03} \quad; \quad b_{02}=b_{03} \, .
\eeq
This also says that 
$a_{\mu \nu}\neq 0$ is just $a_{01}, a_{03} \neq 0$.  
Moreover for each fixed  $\mu$, $\sum_\nu c_{\nu \mu}=0$, so we may select
$c_{01}, c_{02}, c_{12}$  (say) with 
\beq\label{cc}
c_{23}= c_{02}+c_{12} 
\quad ; \quad c_{13}= c_{01}-c_{12} 
\quad ; \quad c_{03}= - c_{02}-c_{01} \, .
\eeq
Next, 
by \eqref{a} above, there are only three $\ell_{\mu \nu}$, 
with 
$\ell_{02}=\ell_{01}\ell_{03}$ out of \eqref{ab}.
Finally thanks to \eqref{cc} and \eqref{ab}, the  six parameters 
$q_{\mu \nu}$ are not all independent:
\beq
q_{23}= q_{02}+q_{12}\ell_{01}
\quad ; \quad 
q_{13}= q_{01}-q_{12} \ell_{01}
\quad ; \quad 
q_{02}= - q_{01}-q_{03} \ell_{01} .
\eeq
Summarizing we are left with 5 parameters 
$$
\ell_{01 }\quad ; \quad \ell_{03}\quad ; \quad q_{01} \quad ; \quad q_{03} 
\quad ; \quad q_{12} 
$$
subject to two conditions
 (obtained by taking quotients of \eqref{cc} by suitable $a_{\mu\nu}$)
\beq
\ell_{03}^2= 1 + q_{03}q_{12} \quad ; \quad 
\ell_{01}^2 + q_{01}^2 + \ell_{01}q_{01} (q_{03}-q_{12})=1 \, .
\eeq
\end{rem}
 
\subsubsection{Example}\label{ex=cdv}  
Take the three angles in $\mathbf{u} \in \IT^3$ to be equal. With notations 
as in \eqref{mLambda} we have that $\lambda_0=1$, $\lambda_1=\lambda_2=\lambda_3=:\lambda \neq \pm 1$. Then all $\ell_{\mu \nu}=1$, while $q_{01}=q_{02}=q_{03}=0$ but the 
$q_{23}, ~q_{13}, ~ q_{12}$ are non zero, and are each proportional to 
$\lambda^{-1}(1-\lambda)$.
The generator $z_0$ is central: $[z_0, z_1]_{-}
=[z_0, z_2]_{-}=[z_0, z_3]_{-}=0, $ while the remaining relations become
\begin{align}\label{cdv-00}
& [z_2,z_3]_{-}= -\lambda^{-1}(1-\lambda) \, z_0 z_1   \, ,  \nn \\
& [z_1,z_3]_{-}= -\lambda^{-1}(1-\lambda) \, z_0 z_2  \, ,  \nn \\
& [z_1,z_2]_{-} = -\lambda^{-1}(1-\lambda) \, z_0 z_3 \, .
\end{align}
This is in analogy with the case in \eqref{skly-k00} for the Sklyanin algebra. 

\subsubsection{$\theta$-deformations}\label{ex:cl} 
The algebra of polynomial functions on the noncommutative four-plane $\IR^4_\theta$
defined in \cite{cl} corresponds to angles $\varphi_0 = \varphi_3 = 0$ and  
$\varphi_1 = \varphi_2  = - \theta /2$ in \eqref{mLambda}. With complex coordinates $\zeta_1 = x_0 + \ii x_3$ 
and $\zeta_2 = x_1 + \ii x_2$ it has commutation relations:
$$
\zeta_1 \zeta_2 = \lambda \zeta_2 \zeta_1 \, , \quad \zeta_1 \zeta_2^* 
= \bar{\lambda} \zeta_2^* \zeta_1  \, , \quad \zeta_1 \zeta_1^* = \zeta_1^* \zeta_1 
\, , \quad \zeta_2 \zeta_2^* = \zeta_2^* \zeta_2 \, ,
$$
together with the conjugated ones, with parameter $\lambda:=e^{i \theta}$. 
These can be written in the form of the present paper, that is as in \eqref{cdv-x} for appropriate parameters $\ell$'s and $p$'s. A direct computation yields:
$\lambda_0=\lambda_3=1$ and $\lambda_1=\lambda_2=e^{i \theta}=\lambda$. In turn:
\begin{align*}
&\ell_{01}=\ell_{02}=  \frac{2 \lambda}{1 + \lambda^2} \, , & \ell_{03}=1 \, ,
\\ 
&p_{01}=p_{13}=-p_{02}=-p_{23}= \frac{1- \lambda^2}{1 + \lambda^2} \, , &p_{03}=p_{12}=0 \, ,
\end{align*}
which requires to take $\theta \neq \frac{\pi}{2}$. 

\subsubsection{Skylanin algebras for  $\k=\IC$}\label{sec:skly-C}
Originally Sklyanin algebras were considered for the case $\k=\IC$.  
Let us return to the family of algebras addressed in \S\ref{sec:skly} above and consider the case $\k=\IC$. If we set $\mathrm{x}_\mu^*=\mathrm{x}_\mu$, for each $\mu=0,1,2,3$, this would not define a $*$-structure: the commutation relations \eqref{skly-k}, in particular those in the right column,  would not be preserved regardless of the choice of $\alpha, \beta, \gamma$. 
This problem can be overcome and a $*$-structure, compatible with the algebra structure, can be introduced  by taking the generator $\mathrm{x}_0$ to be anti-hermitian, that is 
$\mathrm{x}_0^*= -\mathrm{x}_0$. By renaming the generators as 
\beq
{x}_0:= - \ii \mathrm{x}_0 \; , \; {x}_k=: \mathrm{x}_k \; \mbox{ for } k=1,2,3
\eeq  
the relations \eqref{skly-k} can be rewritten in terms of the $x_\mu$ as (cf. \cite[eq.~(32)]{Skly})
\begin{align}\label{skly-C}
[x_0, x_1]_{-} = \ii\alpha [x_2,x_3]_{+} \, , \qquad 
&[x_2,x_3]_{-}  =\ii  [x_0, x_1]_{+} \, ,
\nn \\
[x_0, x_2]_{-} =\ii \beta [x_3,x_1]_{+}  \, , \qquad
& [x_3,x_1]_{-}=\ii   [x_0, x_2]_{+}  \, , 
\nn \\
[x_0, x_3]_{-} = \ii\gamma [x_1,x_2]_{+}  \, , \qquad 
&  [x_1,x_2]_{-}=\ii  [x_0, x_3]_{+} \, 
\end{align}
with $\alpha:=-\mathrm{a}$, $\beta:=-\mathrm{b}$ and $\gamma:=-\mathrm{c}$ still satisfying condition \eqref{cond}: 
\beq\label{cond-C}\alpha + \beta + \gamma +\alpha  \beta  \gamma =0 \, .
\eeq

With this different choice of generators $x_\mu$ for the Sklyanin algebra, the parameters 
$\ell$ and $p$, analogous to those for $\mathrm{x}_\mu$ in \eqref{lp-skly}, are computed to be  
\begin{align}\label{lp-skly-C}
\ell_{01} & :=\frac{1 - \alpha}{1+\alpha} \, , \quad 
p_{01}:= \frac{2 \ii \alpha}{1+\alpha}= \ii( 1-\ell_{01}) \, , \quad
p_{23}:= \frac{2\ii}{1+\alpha}=  \ii (1+\ell_{01} ) 
\nn \\
\ell_{02} & :=\frac{1 + \beta}{1-\beta} \, , \quad
p_{02}:= \frac{2\ii \beta}{1-\beta} = \ii(\ell_{02}-1) \, , \quad
p_{13}:= \frac{-2\ii}{1-\beta}= -\ii( 1+\ell_{02}) 
\nn \\
\ell_{03} & :=\frac{1 - \gamma}{1+\gamma} \, , \quad
p_{03}:= \frac{2 \ii\gamma}{1+\gamma} = \ii(1- \ell_{03}) \, , \quad
p_{12}:= \frac{2\ii}{1+\gamma}=\ii(  1+\ell_{03}) \, ,
\end{align}
which now require $\alpha,\gamma \neq - 1$ and $\beta \neq 1$. 
These new parameters $\ell$ and $p$ still satisfy the three conditions  in 
Definition~\ref{def:lp}, as well as the constraint \eqref{cond-l}.

Taking $\alpha, \beta, \gamma$ to be real (in accordance with the original choice of \cite{Skly}),  
one sees that $\bar \ell_{\mu \nu}= \ell_{\mu \nu}$ for all $\mu,\nu$ and thus that 
$\bar p_{\mu \nu} =-p_{\mu \nu}= p_{\nu \mu}$. Conditions \eqref{*lp} are hence satisfied and 
the choice $x_\mu^*=x_\mu$, for all $\mu$, defines a well-defined $*$-structure on the algebra $\alp$ corresponding to the Sklyanin algebra. 

\begin{rem}  There is quite an overlap between the family of Sklyanin algebras and  
that of Connes--Dubois-Violette algebras described in \S\ref{sec:cdv}. 
For `generic' values of the deformation parameters $\bf u$, both families depend only 
on two parameters \cite[\S 3]{cdv} (cf. also \cite{cdv2}). The $\theta$-deformations of \S\ref{ex:cl} 
are not Sklyanin algebras.
\end{rem}

\section{The exterior algebra of $\alp$}\label{sec:calculus}

There is a natural calculus on the quadratic algebras $\alp$. This section is dedicated to the construction of the Grassmann algebra $(\omlp,\dd)$ of $\alp$.

\subsection{Differential calculus}
We denote by $\omlp=\omlp(\alp)$ the 
unital associative graded $\k$-algebra  $\omlp:=\oplus_{n\in \mathbb{N}} \omlp^n$
 generated by  elements $x_\mu$, $\mu \in \int$, of degree $0$  satisfying relations \eqref{comm-rel-x} and by 
elements $\dd x_\mu$, $\mu \in \int$, of degree $1$ satisfying relations
\beq\label{one-forms}
\dd x_\mu x_\nu = \ell_{\mu \nu} x_\nu \dd x_\mu + p_{\mu\nu} x_{\nu'} \dd x_{\mu'} 
\, ,\quad  
x_\mu \dd x_\nu = \ell_{\mu \nu} \dd x_\nu x_\mu + p_{\mu\nu} \dd x_{\nu'} x_{\mu'} 
\eeq
and
\beq\label{two-forms}
\dd x_\mu \dd x_\nu = - \ell_{\mu \nu} \dd x_\nu \dd x_\mu - p_{\mu\nu} \dd x_{\nu'} \dd x_{\mu'} 
\eeq
for all $\mu,\nu \in \int$. From the properties of the parameters $\ell_{\mu \nu}$, $p_{\mu \nu} $, from \eqref{two-forms} 
one has that $\dd x_\mu \dd x_\mu=0$ for each $\mu$. Also, conditions 
\eqref{one-forms} are consistent in that by substituting the second one in the first one or vice-versa one gets an identity. The same consideration applies to \eqref{two-forms}:  when reusing \eqref{two-forms} in the right hand side it yields an identity.

\begin{rem}
When writing the defining relation \eqref{comm-rel-x} via an $\cR$-matrix as in \eqref{rmrel}, the relations for the forms in \eqref{one-forms} and \eqref{two-forms} can be written as 
\begin{align}\label{rmrelcal}
\dd x_\mu x_\nu &= \sum_{\sigma, \tau} \cR_{\mu \nu}^{~ ~\sigma \tau} x_\sigma \dd x_\tau \, ,  \quad 
x_\mu \dd x_\nu = \sum_{\sigma, \tau} \cR_{\mu \nu}^{~ ~\sigma \tau} \dd x_\sigma x_\tau \, , \\
\dd x_\mu \dd x_\nu & = - \sum_{\sigma, \tau} \cR_{\mu \nu}^{~ ~\sigma \tau} \dd x_\sigma \dd x_\tau \, .
\end{align}
\end{rem}

Next we define the linear operator  $\dd:\alp \to \omlp^1$, by $x_\mu \mapsto \dd x_\mu$ and extend it to  a differential $\dd$ on $\omlp$ by imposing that $\dd^2=0$ and that it satisfies a graded Leibniz rule. From this rule it also follows that each space $\omlp^n$ is an $\alp$-bimodule.

For $\k=\IC$ we further require $\omlp$ to be a $*$-algebra with $*(\dd x_\mu)=\dd (x_\mu^*)=\dd x_\mu$.

\subsection{Higher order forms}
Let us analyse the structure of higher order forms.  Firstly, we observe that relation \eqref{two-forms} follows by any-one of the relations in \eqref{one-forms} by applying $\dd$ and using the graded Leibniz rule. 
Next, by multiplying the relation \eqref{two-forms} on the left or the right by one-forms 
(and using $\dd x_\mu \dd x_\mu=0$), in degree 3 we have several identities.

\subsubsection{Three-forms}
We fix indices $\mu,\nu$. 
By multiplying \eqref{two-forms}  on the left and on the right by all possible one forms, $\dd x_{\mu}$,  $\dd x_{\nu}$
$\dd x_{\mu'}$ and  $\dd x_{\nu'}$, we obtain all the identities that  three-forms have to satisfy. 
The eight equations we obtain are respectively
\begin{align}
&0 = - \ell_{\mu \nu}  \dx_\mu \dx_\nu \dx_\mu - p_{\mu\nu}  \dx_\mu\dx_{\nu'} \dx_{\mu'} 
\label{l1}
\\
&\dx_\nu \dx_\mu \dx_\nu =  - p_{\mu\nu} \dx_\nu \dx_{\nu'} \dx_{\mu'}  \label{l2}
\\
 &\dx_{\mu'} \dx_\mu \dx_\nu = - \ell_{\mu \nu} \dx_{\mu'}  \dx_\nu \dx_\mu - p_{\mu\nu} \dx_{\mu'}  \dx_{\nu'} \dx_{\mu'}  \label{l3}
\\
&\dx_{\nu'} \dx_\mu \dx_\nu = - \ell_{\mu \nu} \dx_{\nu'}  \dx_\nu \dx_\mu  \label{l4}
\end{align}
together with 
\begin{align}
&\dx_\mu \dx_\nu \dx_\mu=  - p_{\mu\nu} \dx_{\nu'} \dx_{\mu'}  \dx_\mu \label{r1}
\\
&0 = - \ell_{\mu \nu} \dx_\nu \dx_\mu \dx_\nu - p_{\mu\nu} \dx_{\nu'} \dx_{\mu'}  \dx_\nu
\label{r2}
\\
&\dx_\mu \dx_\nu \dx_{\mu'} = - \ell_{\mu \nu} \dx_\nu \dx_\mu \dx_{\mu'} \label{r3}
\\
&\dx_\mu \dx_\nu \dx_{\nu'} = - \ell_{\mu \nu} \dx_\nu \dx_\mu \dx_{\nu'} - p_{\mu\nu} \dx_{\nu'} \dx_{\mu'} \dx_{\nu'} \label{r4}
\end{align}

\medskip
\begin{lem}
Suppose that for each pair of indices $(\mu,\nu)$, $\mu,\nu=0,1,2,3$ identities \eqref{l2}, \eqref{l4}, \eqref{r1} and
\eqref{r3} hold, then the remaining identities follow.
\end{lem}

\begin{proof}
First, by using \eqref{l2} and then \eqref{l4} we get \eqref{l1}:
$$
\ell_{\mu \nu}  \dx_\mu \dx_\nu \dx_\mu = \ell_{\mu \nu} p_{\mu\nu} \dx_\mu \dx_{\mu'} \dx_{\nu'} = - p_{\mu\nu}  \dx_\mu\dx_{\nu'} \dx_{\mu'}. 
$$
Next,  by using first \eqref{l4}, property \eqref{lp}, and next \eqref{l2} we get \eqref{l3}:
\begin{align*}
- \ell_{\mu \nu} \dx_{\mu'}  \dx_\nu \dx_\mu &=  \ell_{\mu \nu}^2 \dx_{\mu'}  \dx_\mu \dx_\nu = (1+ p_{\mu \nu} p_{\mu' \nu'}) \dx_{\mu'}  \dx_\mu \dx_\nu \\
&= \dx_{\mu'}  \dx_\mu \dx_\nu  + p_{\mu \nu} \dx_{\mu'}  \dx_{\nu'} \dx_{\mu'} .
\end{align*}
Similarly, by using \eqref{r1} and then \eqref{r3} we promptly obtain \eqref{r2}:
$$
\ell_{\mu \nu}  \dx_\nu \dx_\mu \dx_\nu = \ell_{\mu \nu} p_{\mu\nu} \dx_{\mu'} \dx_{\nu'} \dx_{\nu} = - p_{\mu\nu}  \dx_{\nu'}\dx_{\mu'} \dx_{\nu}. 
$$
Finally, by using first \eqref{r3}, property \eqref{lp}, and then \eqref{r1} we get:
$$
- \ell_{\mu \nu} \dx_{\nu}  \dx_\mu \dx_{\nu'} =  \ell_{\mu \nu}^2 \dx_{\mu}  \dx_\nu \dx_{\nu'}
= \dx_{\mu}  \dx_\nu \dx_{\nu'}  + p_{\mu \nu} \dx_{\nu'}  \dx_{\mu'} \dx_{\nu'} .
$$
that is the last relation.
 \end{proof}
Summing up, three-forms should satisfy, for all pair of indices $(\tau,\sigma)$, the relations
\begin{align}
&\dx_{\tau'} \dx_\sigma \dx_\tau= - \ell_{\tau\sigma} \dx_{\tau'} \dx_{\tau} \dx_\sigma 
\label{three-forms1} \\
&\dx_\tau \dx_\sigma \dx_{\tau'}= - \ell_{\tau\sigma} \dx_{\sigma} \dx_{\tau} \dx_{\tau'}  \label{three-forms2}
\end{align}
together with 
\begin{align}
\dx_\tau \dd x_\sigma \dx_\tau  = - p_{\tau\sigma} \dx_{\sigma'} \dx_{\tau'} \dx_\tau 
 = p_{\tau\sigma} \dx_{\tau} \dx_{\tau'} \dx_{\sigma'}  \, . \label{rep}
\end{align}
It is to be stressed that, despite being $\dd x_\tau \dd x_\tau=0$, in general $\dd x_\tau \dd x_\sigma 
\dd x_\tau \neq 0$; of course these vanish in the classical limit where $p_{\tau \sigma}=0$. 

The relations above, while not all independent (and this will lead to some requirement on the parameters) allow one to find a basis of three-forms. To get a grasp of how this work, let us write explicitly the relations \eqref{three-forms1} and \eqref{three-forms2} that do not contain
$\dd x_0$:  
\begin{align*}
&\dx_3 \dx_1 \dx_2= - \ell_{03} \, \dx_3 \dx_2 \dx_1 
\quad ; \quad
&\dx_2 \dx_1 \dx_3= - \ell_{03} \,\dx_1 \dx_2 \dx_3   
\\
&\dx_2 \dx_1 \dx_3= - \ell_{02} \, \dx_2 \dx_3 \dx_1
\quad ; \quad
&\dx_3 \dx_1 \dx_2= - \ell_{02} \, \dx_1 \dx_3 \dx_2   
\\
&\dx_1 \dx_2 \dx_3= - \ell_{01} \, \dx_1 \dx_3 \dx_2 
\quad ; \quad
&\dx_3 \dx_2 \dx_1= - \ell_{01} \, \dx_2 \dx_3 \dx_1   \, .
\end{align*}
We see that on the left-hand side there do not appear neither the three-form $\dx_1 \dx_3 \dx_2$ nor 
$\dx_2 \dx_3 \dx_1$. This suggests using one of them as the independent one and express the remaining forms as a multiple of the chosen one. With the former $\dx_1 \dx_3 \dx_2$, out of the above relations we get:
\begin{align*}
&\ell_{03} \, \dx_3 \dx_2 \dx_1 = \ell_{02} \, \dx_1 \dx_3 \dx_2   
\quad ; \quad
&\dx_2 \dx_1 \dx_3= \ell_{01} \ell_{03} \,\dx_1 \dx_3 \dx_2   
\\
&\ell_{02} \, \dx_2 \dx_3 \dx_1 = - \ell_{01} \ell_{03} \, \dx_1 \dx_3 \dx_2
\quad ; \quad
&\dx_3 \dx_1 \dx_2= - \ell_{02} \, \dx_1 \dx_3 \dx_2   
\\
&\dx_1 \dx_2 \dx_3= - \ell_{01} \, \dx_1 \dx_3 \dx_2 
\quad ; \quad
&\ell_{01} \ell_{03} \, \dx_2 \dx_3 \dx_1= - \ell_{02} \, \dx_1 \dx_3 \dx_2 \, .
\end{align*}
In fact, for the last one we need to assume that $\ell_{03}\not=0$. 
Next, comparing the second relation in the first column with the last one in the second column we get a condition on the parameters, that is 
$\ell_{02}= \pm \ell_{01} \ell_{03}$; thus the above become
\begin{align*}
&\ell_{03} \, \dx_3 \dx_2 \dx_1 = \pm \ell_{01} \ell_{03} \, \dx_1 \dx_3 \dx_2   
\quad ; \quad
&\dx_2 \dx_1 \dx_3= \ell_{01} \ell_{03} \,\dx_1 \dx_3 \dx_2   
\\
& \pm \ell_{01} \ell_{03} \, \dx_2 \dx_3 \dx_1 = - \ell_{01} \ell_{03} \, \dx_1 \dx_3 \dx_2
\quad ; \quad
&\dx_3 \dx_1 \dx_2= \mp \ell_{01} \ell_{03} \, \dx_1 \dx_3 \dx_2   
\\
&\dx_1 \dx_2 \dx_3= - \ell_{01} \, \dx_1 \dx_3 \dx_2 
\quad ; \quad
&\ell_{01} \ell_{03} \, \dx_2 \dx_3 \dx_1= \mp \ell_{01} \ell_{03} \, \dx_1 \dx_3 \dx_2 \, .
\end{align*}
Had we taken  $\dx_2 \dx_3 \dx_1$ as a basis, we would have obtained an analogous and compatible result (again requiring $\ell_{03}\not=0$). 
To proceed and simplify expressions, we assume that also $\ell_{01}$ is different from zero. Then, the relations on three-forms not containing $\dx_0$ are:
\begin{align}\label{basisthree-0}
\dd x_2 \dx_3 \dx_1 & = \mp \dx_1 \dx_3 \dx_2  \nn \\ 
\dx_1 \dx_2 \dx_3 & = - \ell_{01} \, \dx_1 \dx_3 \dx_2 \nn \\
\dx_3 \dx_2 \dx_1 & = \pm \ell_{01} \, \dx_1 \dx_3 \dx_2  \nn \\ 
\dx_2 \dx_1 \dx_3 & = \ell_{01} \ell_{03} \,\dx_1 \dx_3 \dx_2  \nn \\ 
\dx_3 \dx_1 \dx_2 & = \mp \ell_{01} \ell_{03} \, \dx_1 \dx_3 \dx_2  \,   . 
\end{align}
Next, we list all relations \eqref{rep} whose right hand side does not contain $\dx_0$:
\begin{align*}
\dx_1 \dd x_0 \dx_1 & = p_{01} \, \dx_2 \dx_3 \dx_1 = - p_{01} \, \dx_1 \dx_3 \dx_2    \\
\dx_2 \dd x_0 \dx_2 & = p_{02} \, \dx_1 \dx_3 \dx_2 = - p_{02} \, \dx_2 \dx_3 \dx_1    \\
\dx_3 \dd x_0 \dx_3 & = p_{03} \, \dx_1 \dx_2 \dx_3 = - p_{03} \, \dx_3 \dx_2 \dx_1   
\end{align*}
and using \eqref{basisthree-0} we arrive at 
\begin{align}\label{basisthree-00b}
\dx_1 \dd x_0 \dx_1 & = \mp p_{01} \, \dx_1 \dx_3 \dx_2  = - p_{01} \, \dx_1 \dx_3 \dx_2 \nn  \\
\dx_2 \dd x_0 \dx_2 & = p_{02} \, \dx_1 \dx_3 \dx_2  = \pm p_{02} \, \dx_1 \dx_3 \dx_2  \nn  \\
\dx_3 \dd x_0 \dx_3 & = - \ell_{01} p_{03} \, \dx_1 \dx_3 \dx_2 = \mp \ell_{01} p_{03} \, \dx_1 \dx_3 \dx_2 \, . 
\end{align}
From this, we see that the choice $\ell_{02}= - \ell_{01} \ell_{03}$ will lead to 
$$
p_{01} = p_{02} = p_{03} = 0 ,
$$
while the choice $\ell_{02}=\ell_{01} \ell_{03}$ yields
\begin{align}\label{basisthree-0b}
\dx_1 \dd x_0 \dx_1 & = - p_{01} \, \dx_1 \dx_3 \dx_2  \nn  \\
\dx_2 \dd x_0 \dx_2 & = p_{02} \, \dx_1 \dx_3 \dx_2   \nn  \\
\dx_3 \dd x_0 \dx_3 & = - \ell_{01} p_{03} \, \dx_1 \dx_3 \dx_2 \, . 
\end{align}

Equations \eqref{basisthree-0} and \eqref{basisthree-0b} list all three-forms that can be expressed in terms of the three-form $\theta_0 := \dx_1 \dx_3 \dx_2$. We can repeat the analysis above for each index $\nu \in \{0,1,2,3\}$. 
For this it is convenient to tabulate all possible values of the indices. 
For each index $\nu \in \int$ fixed, we define indices $\tilde\nu, \mu_\nu, \tilde\mu_\nu$ by
\beq\label{tabella}
\begin{array}{c|c|c|c}
\nu & \tilde\nu & \mu_\nu & \tilde\mu_\nu 
\\
\hline 0&1&2&3
\\
1&0&3&2
\\
2&3&0&1
\\
3&2&1&0
\end{array} \, .
\eeq
These are such that for $(\nu,\mu_\nu)$ fixed, then $(\nu',\mu_\nu')=(\tilde\nu, \tilde\mu_\nu)$. 
This then gives:
\beq\label{primes}
(\nu,\mu_\nu)'=(\tilde\nu, \tilde\mu_\nu)\, , \quad 
(\nu,\mt)'=(\tilde\nu,\mu_\nu) \quad \mbox{and} \quad 
(\nu,\tilde\nu)'=(\mu_\nu,\mt).
\eeq
Furthermore, an explicit computation for each $\nu \in \int$ fixed yields that 
\beq\label{lem1} 
\{ (\mu, \nu',\mu')\}_{\mu \neq \nu}=\{(\mn, \nt, \mt), (\mt, \nt, \mn),(\nt, \mn,\mt)\} .
\eeq
Finally, using the table \eqref{tabella} again with a direct computation one finds
for each index $\nu$ 
\beq\label{ells}
\ell_{\mn \nu}= \ell_{02} \; , \quad
\ell_{\mn \mt}= \ell_{01} \; , \quad 
\ell_{\mn \nt}= \ell_{03} \; , \qquad \forall \nu = 0,1,2,3 \, .
\eeq

In the relations \eqref{three-forms1} and \eqref{three-forms2}, the ones not containing $\dx_\nu$ are those for indices $\sigma,\tau$ such that neither $\sigma$ nor $\tau$ are equal to $\nu$ and also such that 
$\tau' \neq \nu$. 
The first assumption gives for the pair $(\tau,\sigma)$ that
$$
(\tau,\sigma) = (\nt,\mn), ~ (\nt,\mt), ~ (\mn,\mt), ~ (\mn,\nt), ~ (\mt,\nt), ~ (\mt,\mn)
$$
but from \eqref{primes} each of the first three cases gives $\tau' = \nu$ and so it has also to be excluded.
Thus, only the last three choices are possible and \eqref{three-forms1} and \eqref{three-forms2} give 
the following six equations:
\begin{align*}
&\dx_{\mt} \dx_\nt \dx_\mn= - \ell_{\mn\nt} \dx_{\mt} \dx_{\mn} \dx_\nt 
\quad ; \quad
&\dx_\mn \dx_\nt \dx_{\mt}= - \ell_{\mn\nt} \dx_{\nt} \dx_{\mn} \dx_{\mt}   
\\
&\dx_{\mn} \dx_\nt \dx_\mt= - \ell_{\mt\nt} \dx_{\mn} \dx_{\mt} \dx_\nt 
\quad ; \quad
&\dx_\mt \dx_\nt \dx_{\mn}= - \ell_{\mt\nt} \dx_{\nt} \dx_{\mt} \dx_{\mn}   
\\
&\dx_{\nt} \dx_\mn \dx_\mt= - \ell_{\mt\mn} \dx_{\nt} \dx_{\mt} \dx_\mn 
\quad ; \quad
&\dx_\mt \dx_\mn \dx_{\nt}= - \ell_{\mt\mn} \dx_{\mn} \dx_{\mt} \dx_{\nt}  \, .
\end{align*}
Now on the left-hand side there do not appear neither the three-form $\dx_{\nt} \dx_{\mt} \dx_\mn$ nor 
$\dx_\mn \dx_{\mt} \dx_{\nt}$ and we can use one of them as the independent one. With  the former, if we denote $\theta_\nu:=\dx_{\nt} \dx_{\mt} \dx_\mn$, the above become
\begin{align*}
&\dx_{\mt} \dx_\nt \dx_\mn= - \ell_{\mn\nt} \dx_{\mt} \dx_{\mn} \dx_\nt 
\quad ; \quad
&\dx_\mn \dx_\nt \dx_{\mt}=  \ell_{\mn\nt} \ell_{\mt\mn} \theta_\nu  
\\
&\dx_{\mn} \dx_\nt \dx_\mt= - \ell_{\mt\nt} \dx_{\mn} \dx_{\mt} \dx_\nt 
\quad ; \quad
&\dx_\mt \dx_\nt \dx_{\mn}= - \ell_{\mt\nt} \theta_\nu 
\\
&\dx_{\nt} \dx_\mn \dx_\mt= - \ell_{\mt\mn} \theta_\nu 
\quad ; \quad
&\dx_\mt \dx_\mn \dx_{\nt}= - \ell_{\mt\mn} \dx_{\mn} \dx_{\mt} \dx_{\nt}   
\end{align*}
and next
\begin{align*}
&\ell_{02} \theta_\nu =  \ell_{03} \dx_{\mt} \dx_{\mn} \dx_\nt 
\quad ; \quad
&\dx_\mn \dx_\nt \dx_{\mt}=  \ell_{03} \ell_{01} \theta_\nu  
\\
&\ell_{03} \ell_{01} \theta_\nu= - \ell_{02} \dx_{\mn} \dx_{\mt} \dx_\nt 
\quad ; \quad
&\dx_\mt \dx_\nt \dx_{\mn}= - \ell_{02} \theta_\nu 
\\
&\dx_{\nt} \dx_\mn \dx_\mt= - \ell_{01} \theta_\nu 
\quad ; \quad
&\dx_\mt \dx_\mn \dx_{\nt}= - \ell_{01} \dx_{\mn} \dx_{\mt} \dx_{\nt}   \, .
\end{align*}
Again we need $\ell_{03}\not=0$  and thus $\ell_{02} = \pm \ell_{01} \ell_{03}$. 
Then, with $\ell_{01}$ different from zero we get
\begin{align}\label{basis3}
&\dx_{\nt} \dx_{\mt} \dx_\mn  := \theta_\nu
\nn \\
& \dx_{\mn} \dx_{\mt} \dx_\nt = \mp \theta_\nu
\nn \\
&\dx_{\mt} \dx_{\mn} \dx_\nt  = \pm \ell_{01} \theta_\nu 
\nn \\
&\dx_{\nt} \dx_\mn \dx_\mt= - \ell_{01} \theta_\nu 
\nn \\
&\dx_\mn \dx_\nt \dx_{\mt}=  \ell_{01} \ell_{03} \theta_\nu
\nn \\
&\dx_\mt \dx_\nt \dx_{\mn}= \mp \ell_{01} \ell_{03} \theta_\nu \, .
\end{align}
Next, the relations in \eqref{rep} not containing $\dd x_\nu$ in the right hand side are
\begin{align*}
\dd x_\nt \dd x_\nu \dd x_\nt & 
= p_{\nu \nt} \, \dd x_\mn \dd x_\mt \dd x_\nt = - p_{\nu \nt} \, \dd x_\nt \dd x_\mt \dd x_\mn  \\
\dd x_\mn \dd x_\nu \dd x_\mn & 
= p_{\nu \mn} \, \dd x_\nt \dd x_\mt \dd x_\mn = -p_{\nu \mn} \, \dd x_\mn \dd x_\mt \dd x_\nt   \\
\dd x_\mt \dd x_\nu \dd x_\mt & 
= p_{\nu \mt} \, \dd x_\nt \dd x_\mn \dd x_\mt = - p_{\nu \mt} \, \dd x_\mt \dd x_\mn \dd x_\nt \ . \\
\end{align*}
And using \eqref{basis3} we arrive at 
\begin{align*}
\dd x_\nt \dd x_\nu \dd x_\nt & 
= \mp \, p_{\nu \nt} \, \theta_\nu = - p_{\nu \nt} \, \theta_\nu \, , \\
\dd x_\mn \dd x_\nu \dd x_\mn & 
= p_{\nu \mn} \, \theta_\nu = \pm \, p_{\nu \mn} \, \theta_\nu \, , \\
\dd x_\mt \dd x_\nu \dd x_\mt & 
= - \ell_{01} p_{\nu \mt} \, \theta_\nu = \mp \, \ell_{01} p_{\nu \mt} \, \theta_\nu \, . \\
\end{align*}
As before the choice $\ell_{02}= - \ell_{01} \ell_{03}$ leads to 
\beq\label{ptutti0}
p_{\nu \nt} = p_{\nu \mn} = p_{\nu \mt} = 0 ,
\eeq
while the choice $\ell_{02}=\ell_{01} \ell_{03}$ yields
\begin{align}\label{basis3b}
\dd x_\nt \dd x_\nu \dd x_\nt & 
= - p_{\nu \nt} \, \theta_\nu ,  \nn \\
\dd x_\mn \dd x_\nu \dd x_\mn & 
= p_{\nu \mn} \theta_\nu \, , \nn \\
\dd x_\mt \dd x_\nu \dd x_\mt & 
= - \ell_{01} p_{\nu \mt} \, \theta_\nu  \, .  
\end{align}
We may conclude that the space $\omlp^3$ of three-forms is generated as a bi-module by the four elements
$\theta_\nu=\dx_{\nt} \dx_{\mt} \dx_\mn$, for $\nu\in\{0,1,2,3\}$. 

\subsubsection{Four-forms}\label{sec:4forms}
We move to the analysis of the bi-module $\omlp^4$ of four-forms. In this section we take 
all $\ell_{\mu \nu}$ to be non zero, $\ell_{01},\ell_{02}, \ell_{03} \neq 0$ and (avoiding the case where all $p_{\mu\nu}$ vanish, see \eqref{ptutti0} above)
\beq\label{hyp}
\ell_{02}=\ell_{01} \ell_{03} \; .
\eeq 
These are  the natural assumptions in order to include the classical commutative case 
(where $\ell_{\mu \nu} =1$ for all $\mu,\nu$) and they, in particular \eqref{hyp}, are satisfied by the Connes--Dubois-Violette four-planes (cf. after \eqref{cc}) and Sklyanin algebras (cf. \eqref{cond-l}).

Firstly, since for a fixed index $\nu$ all three-forms that do not contain $\dx_\nu$ 
are proportional to $\theta_\nu$, we observe that
$$\dx_\tau  \theta_\nu=0 = \theta_\nu \dx_\tau  \quad \mbox{ for all } \tau \neq \nu \, .
$$
Hence, a priori, as candidates for basis elements we need to analyse only the four-forms
$$
\dx_\nu  \theta_\nu \quad \mbox{and} \quad \theta_\nu \dx_\nu  \quad \mbox{ for all } \nu=0,1,2,3 \, .
$$
We show that all these forms are proportional and, as a consequence, the bi-module $\omlp^4$
of four-forms is one-dimensional.  Out of \eqref{basis3}, we observe that 
\beq\label{oss-3forms}
\dx_\mu \dx_\tau \dx_\sigma = - \dx_\sigma \dx_\tau \dx_\mu \;
\mbox{ for all distinct indices }  \mu,\tau,\sigma \, .
\eeq
Using this result  twice, we promptly obtain
\begin{align*}
\omega_3= \dx_3 \theta_3 = \dx_3 (\dx_2 \dx_0 \dx_1) =
- (\dx_3 \dx_1 \dx_0) \dx_2 =   \dx_0 \dx_1 \dx_3 \dx_2= \omega_0
\\
 \omega_2= \dx_2 \theta_2 = \dx_2 (\dx_3 \dx_1 \dx_0) =
- (\dx_2 \dx_0 \dx_1) \dx_3 =  \dx_1 \dx_0 \dx_2 \dx_3= \omega_1 \, .
\end{align*}
Moreover, by using  \eqref{two-forms}, we compute
\begin{align*}
\omega_1&= (\dx_1 \dx_0)(\dx_2 \dx_3) = 
(\ell_{01} \dx_0 \dx_1 +p_{10} \dx_2 \dx_3) (\ell_{01} \dx_3 \dx_2 +p_{23} \dx_1 \dx_0)
\\
&=
\ell_{01}^2 \dx_0 \dx_1 \dx_3 \dx_2 - p_{01}p_{23} \dx_2 \dx_3 \dx_1 \dx_0
\\
&=
\ell_{01}^2 \omega_0 - p_{01}p_{23} \omega_2.
\end{align*}
Thus, $\ell_{01}^2 \omega_0 = (1+p_{01}p_{23} ) \omega_1 = \ell_{01}^2 \omega_1$ (being $(1+p_{01}p_{23} )= \ell_{01}^2$ from condition \eqref{lp} of Definition~\ref{def:lp})
giving the further identity
$\ell_{01}^2 \omega_1  =\ell_{01}^2 \omega_0$
and hence, being $\ell_{01}\neq 0$,  $\omega_1  =\omega_0$.
Summarizing we have found that
\beq \label{omega}
\omega_0  =\omega_1=\omega_2  =\omega_3=:\omega = \dx_0 \dx_1 \dx_3 \dx_2 \, .
\eeq
We next show that for each index $\nu$, the form
$\theta_\nu \dx_\nu$ is proportional to $\omega$ too.
By using \eqref{oss-3forms} we easily compute 
\begin{align*}
\theta_0 \dx_0 = \dx_1(\dx_3 \dx_2 \dx_0) = -  \dx_1(\dx_0 \dx_2 \dx_3) =-\omega_1 =-\omega
\\
\theta_1 \dx_1 = \dx_0(\dx_2 \dx_3 \dx_1) = -  \dx_0(\dx_1 \dx_3 \dx_2) =-\omega_0 =-\omega
\\
\theta_2 \dx_2 = \dx_3(\dx_1 \dx_0 \dx_2) = -  \dx_3(\dx_2 \dx_0 \dx_1) =-\omega_3 =-\omega
\\
\theta_3 \dx_3 = \dx_2(\dx_0 \dx_1 \dx_3) = -  \dx_2(\dx_3 \dx_1 \dx_0) =-\omega_2 =-\omega
\end{align*}
as wished. This also shows that $ \theta_\nu \dx_\nu = -\dx_\nu\theta_\nu$ for each $\nu$.
We can thus conclude that the bi-module of four-forms is one-dimensional, 
generated by the form $\omega$ (say). This top form $\omega$ will be shown to be not zero in \S\ref{vofo} by identifying it with the (differential calculus representation of the) volume form of a pre-regular multilinear form for our family of algebras $\alp$.

As a final remark, it is worth stressing to observe that in $\omlp^4$, 
in addition to `usual' forms $\dx_\nu \dx_\mu \dx_\tau \dx_\sigma$ 
with $\nu \neq \mu \neq \tau \neq \sigma$, there are also `quantum' elements of the form
$\dx_\nu \dx_\mu \dx_\nu \dx_\mu$,  with $\nu \neq \mu$. Nevertheless these forms are proportional to $\omega$, accordingly to the following relations deduced from \eqref{basis3b}:
\begin{align*}
\dx_\nu  \dx_\nt \dx_\nu \dx_\nt  &= - p_{\nu \nt} \, \omega_\nu \, , \\
\dx_\nu  \dx_\mn \dx_\nu \dx_\mn  &= p_{\nu \mn} \omega_\nu \, , \\
\dx_\nu \dx_\mt \dx_\nu \dx_\mt  &= - \ell_{01} p_{\nu \mt} \, \omega_\nu \, .  
\end{align*}
Thus they vanish in the `classical' commutative case all $p_{\mu \nu}$ being zero then.

\subsection{The volume form} \label{vofo}
In this section we shall make contact with the theory of pre-regular multilinear forms 
of \cite{dv05,dv07}. 
Let $W$ be the linear form on $\k^4$ with components  
\beq\label{W}
W(e_{\alpha},e_{\beta},e_{\rho},e_{\sigma})=:
W_{\alpha \beta \rho \sigma} = 
L_{\alpha \beta \rho \sigma} \varepsilon_{\alpha \beta \rho \sigma} + P_{\alpha \beta} \delta_{\alpha \rho} \delta_{\beta \sigma}  \varepsilon_{\alpha \beta \alpha' \beta'}
\eeq
in the canonical basis $\{e_{\mu}, \mu=0,1,2,3\}$ of $\k^4$,
where $\varepsilon_{\alpha \beta \rho \sigma} $ is the completely antisymmetric tensor with $\varepsilon_{0123}=1 $ and where, for distinct indices $\alpha, \beta, \rho, \sigma$, the components $L_{\alpha \beta \rho \sigma}$ and $P_{\alpha \beta}$ are determined (uniquely)  by the properties
\begin{enumerate}
\item \quad
$L_{\alpha \beta \rho \sigma} =L_{\sigma \alpha \beta \rho }$  \quad ; 
\quad $L_{\alpha \beta \rho \sigma}  = L_{\beta  \alpha  \sigma \rho } $ \,\, ;  \vspace{5pt}
\item \quad
$L_{\alpha \beta \alpha' \beta'} = \ell_{\beta' \alpha' } L_{\alpha \beta  \beta' \alpha'}$  \,\, ; \vspace{5pt}
\item \quad
$P_{\alpha \beta}= p_{\beta' \alpha'} L_{\alpha \beta  \beta' \alpha'}$
\end{enumerate}
and by setting $L_{0132} :=1$. One easily shows that
\beq\label{Pprop}
P_{ \beta \alpha}=
p_{ \alpha' \beta'} L_{ \beta \alpha  \alpha'\beta' }
=- p_{\beta' \alpha'} L_{\alpha \beta  \beta' \alpha'}=- P_{  \alpha \beta} \, .
\eeq

In relation to \cite[Def.~2]{dv07} we have the following result
\begin{lem}
The linear form $W$ is pre-regular (without twist), that is   
\begin{enumerate}[(I)]
\item \quad
$W_{\sigma \alpha \beta \rho }  = - W_{\alpha \beta \rho \sigma} $ 
\qquad \mbox{for all indices} \quad $\alpha, \beta, \rho, \sigma$
\item[~]
\item \quad
If $v \in \k^4$ is such that  
$W(v,e_{\beta},e_{\rho},e_{\sigma}) = 0$ for all indices $\beta, \rho, \sigma$, then $v=0$. 
\end{enumerate}
\end{lem}
\begin{proof}
By using the defining properties (1), (2) and (3) and in particular \eqref{Pprop}, for all indices 
$\alpha, \beta, \rho, \sigma$ one verifies that
\begin{align*}
W_{\sigma \alpha \beta \rho } & = L_{\sigma \alpha \beta \rho } \varepsilon_{\sigma \alpha \beta \rho } + P_{\sigma \alpha }  \delta_{\sigma \beta }  \delta_{ \rho \alpha} \varepsilon_{\sigma \alpha  \sigma' \alpha' } = - L_{\alpha \beta \rho \sigma} \varepsilon_{\alpha \beta \rho \sigma} - P_{  \alpha \beta} \delta_{\alpha \rho} \delta_{\beta \sigma}  \varepsilon_{\alpha \beta \alpha' \beta'}
\\ & = - W_{\alpha \beta \rho \sigma} \, ,
\end{align*}
showing that $W$ is cyclic. Next, suppose there is a vector 
$v = (v_\alpha ) \in \k^4$ such that for all indices $\beta, \rho, \sigma$ it holds that 
$$
0 = W(v,e_{\beta},e_{\rho},e_{\sigma})= \sum_\alpha v_\alpha \, W_{\alpha \beta \rho \sigma} = 
\sum_\alpha v_\alpha\, \big(L_{\alpha \beta \rho \sigma} \varepsilon_{\alpha \beta \rho \sigma} + P_{\alpha \beta} \delta_{\alpha \rho} \delta_{\beta \sigma}  \varepsilon_{\alpha \beta \alpha' \beta'} \big) .
$$
Then, from the properties of the $L$'s and $P$'s before one gets that $v=0$; thus $W$ 
is 1-site non-degenerate. The two properties $(I)$ and $(II)$ say that $W$ is pre-regular. 
\end{proof}

\begin{lem}
Let $A(W,2)$ be the quadratic algebra generated by elements $x_\mu$, $\mu=0,1,2,3$,
with relations
\beq\label{comm-W}
\sum_{\rho, \sigma} W_{\alpha \beta \rho \sigma} \, x_\rho x_\sigma =0 , \quad  \forall ~ \alpha , \beta =0,1,2,3 \,\, .
\eeq
Then $A(W,2)$ coincides with the algebra $\alp$.
\end{lem}
\begin{proof}
By the antisymmetry of $\varepsilon_{\alpha \beta \rho \sigma}$, 
fixing $\alpha, \beta$, the only possibilities 
for the last pair of indices in $L_{\alpha \beta \rho \sigma}$ are $(\rho, \sigma) = (\beta', \alpha')$ or $(\rho, \sigma) = (\alpha', \beta')$. Moreover, taking $\ell_{01}, \ell_{02}\neq 0$ one has that $L_{\alpha \beta \beta' \alpha'} \neq 0$ for all $ \alpha, \beta$. Then, for $\alpha, \beta$ (arbitrary but) fixed we have
\begin{align*}
\sum_{\rho \sigma} W_{\alpha \beta \rho \sigma} x_\rho x_\sigma =0 
&\iff
\varepsilon_{\alpha \beta \alpha' \beta'} \left(
L_{\alpha \beta \alpha' \beta'} x_{\alpha'} x_{\beta'}
- L_{\alpha \beta \beta' \alpha'} x_{\beta'} x_{\alpha'}
 + P_{\alpha \beta} x_{\alpha } x_{\beta}  
\right)=0
\\
&\iff
\varepsilon_{\alpha \beta \alpha' \beta'}  L_{\alpha \beta \beta' \alpha'} \left(
 \ell_{ \beta' \alpha'} x_{\alpha'} x_{\beta'}
-  x_{\beta'} x_{\alpha'}
 + p_{ \beta' \alpha'} x_{\alpha } x_{\beta}  
\right)=0
\\ \\
&\iff x_{\beta'} x_{\alpha'}=
 \ell_{ \beta' \alpha'} x_{\alpha'} x_{\beta'}
 + p_{ \beta' \alpha'} x_{\alpha } x_{\beta}  \, ,
\end{align*}
showing that the generators elements $x_\mu$ satisfy conditions \eqref{comm-W} if and only if  they satisfy \eqref{comm-rel-x}. Thus the algebras $A(W,2)$ and $\alp$ are the same.
\end{proof}

By the theory of pre-regular forms the element $1 \ot W$ is a nontrivial Hochschild cycle 
on $A(W,2)$ (see e.g. \cite[Prop. 10]{dv07}). 
We are lead to define as a volume form 
the four-form
\begin{align}\label{w0}
\mbox{vol} & := \sum_{\alpha, \beta ,\rho, \sigma} W_{\alpha \beta \rho \sigma} \dd x_\alpha \dd x_\beta \dd x_\rho \dd x_\sigma 
\nn \\
& = \sum_{\alpha, \beta ,\rho, \sigma}
L_{\alpha \beta \rho \sigma}  ~ \varepsilon_{\alpha \beta \rho \sigma} \dd x_\alpha \dd x_\beta \dd x_\rho \dd x_\sigma  
+ \sum_{\alpha, \beta  } P_{\alpha \beta} ~  \varepsilon_{\alpha \beta \alpha' \beta'} \dd x_\alpha \dd x_\beta \dd x_\alpha \dd x_\beta ~.
\end{align}

\begin{rem} The family $\IR^4_\mathbf{u}$ of nocommutative four-planes
introduced in \cite{cdv} and described briefly in \S\ref{sec:cdv} was obtained in connection with a problem in 
$K$-homology. In particular, out of the top Chern class of a unitary there was defined a Hochschild cycle playing the role of the volume form of $\IR^4_\mathbf{u}$. This cycle is of the form (cf. \cite[eq. (2.14)]{cdv}). 
$$
v = \sum_{\alpha, \beta ,\rho, \sigma} 
\varepsilon_{\alpha \beta \rho \sigma} S_{\alpha \beta \rho \sigma} \, \II \ot x_\alpha \ot x_\beta \ot x_\rho \ot x_\sigma
-\sum_{\alpha, \beta } T_{\alpha \beta} \, \II \ot x_\alpha \ot x_\beta \ot x_\alpha \ot x_\beta \, , 
$$
with explicit tensors $S_{\alpha \beta \rho \sigma}$ and $T_{\alpha \beta}$ which depend on  
the deformation parameters $\mathbf{u} \in \IT^3$.  
A comparison with the volume form in \eqref{w0} (for the algebras $\alp$ of \S \ref{sec:cdv}) shows that the latter is a differential calculi representation of the homology class $v$. 
\end{rem}

\subsubsection{Explicit expression of the volume form}
Let us have a closer look at the components $W_{\alpha \beta \rho \sigma}$ of  $\mbox{vol}$. 
Firstly, by the properties $L_{\alpha \beta \rho \sigma} =L_{\sigma \alpha \beta \rho }  = L_{\beta  \alpha  \sigma \rho } $, we have that
\begin{align}\label{L-expl} 
L_{0132}=L_{2013}=L_{3201}=L_{1320}=L_{1023}=L_{3102}=L_{2310}=L_{0231} \nn
\\
L_{0123}=L_{3012}=L_{2301}=L_{1230}=L_{1032}=L_{0321}=L_{3210}=L_{2103} 
\nn
\\
L_{0312}=L_{2031}=L_{1203}=L_{3120}=L_{3021}=L_{1302}=L_{2130}=L_{0213}
\end{align}
with
$$
L_{0132} =1 ~, \quad 
L_{0123}= \ell_{01} ~, \quad 
L_{0312}= \ell_{01}\ell_{03} = \ell_{02}
$$
from the properties $L_{\alpha \beta \alpha' \beta'} = \ell_{\beta' \alpha' } L_{\alpha \beta  \beta' \alpha'}$ and  the condition $L_{0132} =1$.

\noindent
Next, being $P_{\alpha \beta}= p_{\beta' \alpha'} L_{\alpha \beta  \beta' \alpha'}$, we have
\begin{align}\label{P-expl} 
& 
P_{01}=p_{32} ~,  \quad P_{02}= p_{31} ~, \quad P_{03}= \ell_{01}p_{21} ~, \nn \\
& 
P_{13}=p_{20} ~, \quad P_{23}=p_{10} ~, \quad P_{12}=\ell_{01} p_{30} ~,
\end{align}
with $P_{\alpha \beta}= - P_{ \beta \alpha}$.

On the other hand,  we have shown in \S \ref{sec:4forms} that all $4$-forms are proportional. 
In particular, for distinct indices $\alpha, \beta,\rho, \sigma$ we have found that 
$$
\dx_\alpha \dx_\beta \dx_\rho \dx_\sigma= \eta_{\alpha \beta \rho \sigma} \, \omega
$$ where $\omega$ is the generator of $\omlp^4$ introduced in \eqref{omega} and explicitly:
\begin{align}\label{eta-l} 
\eta_{\nu \nt \mt \mn}= 
&~ \eta_{0132}= \eta_{1023}=  \eta_{2310}=  \eta_{3201}= 1 = - \epsilon_{\nu \nt \mt \mn} 
\nn
\\
\eta_{\nu \mn \mt \nt}=
&~ \eta_{0231}= \eta_{1320}=  \eta_{2013}=  \eta_{3102}= -1 = - \epsilon_{\nu \mn \mt \nt}
\nn
\\
\eta_{\nu \mt \mn \nt}=
&~ \eta_{0321}= \eta_{1230}=  \eta_{2103}=  \eta_{3012}= \ell_{01} =
- \ell_{01}\epsilon_{\nu \mt \mn \nt}
\nn
\\
\eta_{\nu \nt \mn \mt}=
&~ \eta_{0123}= \eta_{1032}=  \eta_{2301}=  \eta_{3210}= - \ell_{01}  = - \ell_{01} \epsilon_{\nu \nt \mn \mt}
\nn
\\
\eta_{\nu \mn \nt \mt}=
&~ \eta_{0213}= \eta_{1302}=  \eta_{2031}=  \eta_{3120}= \ell_{01} \ell_{03} =
- \ell_{01} \ell_{03}  \epsilon_{\nu \mn \nt \mt}
\nn
\\
\eta_{\nu \mt \nt \mn}=
&~\eta_{0312}= \eta_{1203}=  \eta_{2130}=  \eta_{3021}= -\ell_{01} \ell_{03}
= -\ell_{01} \ell_{03}  \epsilon_{\nu \mt \nt \mn}
\end{align}
with, as above, $\epsilon_{\alpha \beta \rho\sigma}$ the completely antisymmetric tensor with $\epsilon_{0123}=1$.
Moreover, for all $\alpha \neq \beta$ we have found
$$
\dx_\alpha \dx_\beta \dx_\alpha \dx_\beta= \eta_{\alpha \beta \alpha \beta}\omega ~,
$$ where now
\begin{align}\label{eta-p}
\eta_{\nu \nt \nu \nt}= -p_{\nu \nt}  \quad  \Rightarrow \quad 
&\eta_{0101}= -\eta_{1010}= - p_{01}\quad
&& \eta_{2323}= -\eta_{3232}=-p_{23} 
\nn
\\
\eta_{\nu \mn \nu \mn}= p_{\nu \mn} \quad  \Rightarrow \quad 
&\eta_{0202}=- \eta_{2020}= p_{02} \quad 
&& \eta_{1313}= -\eta_{3131}= p_{13}
\\
\eta_{\nu \mt \nu \mt}=-\ell_{01}p_{\nu \mt} \quad  \Rightarrow\quad 
&\eta_{0303}= -\eta_{3030}= -\ell_{01}p_{03} \quad 
&& \eta_{1212}=-\eta_{2121}=-\ell_{01}p_{12} \, . \nn
\end{align}

Hence the volume form $\mbox{vol}$ in \eqref{w0}
is  proportional to the  generator $\omega$ too. 
We next determine the  explicit coefficient of proportionality. By a comparison between  
\eqref{L-expl} with \eqref{eta-l} and between \eqref{P-expl} with \eqref{eta-p},
we observe that 
\beq
\eta_{\alpha \beta \rho \sigma} =- L_{\alpha \beta \rho \sigma} ~\varepsilon_{\alpha \beta \rho \sigma}
\quad ; \quad
\eta_{\alpha \beta \alpha \beta} =- P_{ \beta' \alpha' } ~\varepsilon_{\alpha \beta \alpha' \beta'}
\eeq
for all distinct indices $\alpha, \beta, \rho, \sigma$. We thus compute
\begin{align*}
\mbox{vol} 
&= \sum_{\alpha, \beta ,\rho, \sigma}
L_{\alpha \beta \rho \sigma}  ~ \varepsilon_{\alpha \beta \rho \sigma} \dd x_\alpha \dd x_\beta \dd x_\rho \dd x_\sigma  
+ \sum_{\alpha, \beta  } P_{\alpha \beta} ~  \varepsilon_{\alpha \beta \alpha' \beta'} \dd x_\alpha \dd x_\beta \dd x_\alpha \dd x_\beta 
\\
&= - \sum_{\alpha, \beta ,\rho, \sigma}
L^2_{\alpha \beta \rho \sigma}  ~ \varepsilon^2_{\alpha \beta \rho \sigma}  ~ \omega
- \sum_{\alpha, \beta  } P_{\alpha \beta} P_{ \beta' \alpha' } ~  \varepsilon^2_{\alpha \beta \alpha' \beta'}   ~ \omega
\\
&= - \sum_{\alpha \neq \beta} \big( 
L^2_{\alpha \beta \alpha' \beta'}  +  L^2_{\alpha \beta  \beta' \alpha'} 
+  P_{\alpha \beta} P_{ \beta' \alpha' }   \big) \omega
\\
&= - \sum_{\alpha \neq \beta} L^2_{\alpha \beta  \beta' \alpha'}  \big( 
\ell^2_{\alpha \beta}  +  1 
+  p_{ \beta' \alpha' }  p_{\alpha \beta}  \big) \omega
\\
&= - 2\sum_{\alpha \neq \beta} L^2_{\alpha \beta  \beta' \alpha'}~ \omega
\\
&= - 2 (8 + 4 \ell_{01}^2) ~ \omega
\end{align*}
where in the last equality we have used \eqref{L-expl} and the relation $\ell^2_{\alpha \beta}  
+  p_{ \beta' \alpha' }  p_{\alpha \beta}=1$.

\section{The quantum spheres $\slp$}\label{sec:spheres}

In this section we introduce quantum three-spheres as quotients of the algebras $\alp$.  To do that we study the center of the algebras $\alp$ and show that, under suitable conditions for the parameters $\ell_{\mu \nu}$ and $p_{\mu \nu}$, the quadratic element $R:= \sum x_\mu^2$ is central in $\alp$. 

\noindent
We start with the following preliminary result.
\begin{lem}\label{lem2}
For all $\mu,\nu \in \int$ fixed, it holds that:
\beq\label{xx}
x_{\mu} x_{\mu'} x_{\nu'} + x_{\nu'} x_{\mu'} x_{\mu} = \ell_{\mu \nu} \left( x_{\mu} x_{\nu'} x_{\mu'} +   x_{\mu'} x_{\nu'} x_{\mu} \right) .
\eeq
\end{lem}
\begin{proof}
For this statement, we use \eqref{comm-rel-x}, that is
$x_\mu x_\nu = \ell_{\mu \nu} x_\nu x_\mu + p_{\mu\nu} x_{\nu'} x_{\mu'}$. 
Firstly, multiply it on the left by $x_{\mu'}$, thus getting
$$
x_{\mu'} x_\mu x_\nu = \ell_{\mu \nu} x_{\mu'} x_\nu x_\mu + p_{\mu\nu} x_{\mu'} x_{\nu'} x_{\mu'} \, .
$$
Then, exchange $\mu \leftrightarrow \nu$ and multiply it by $x_{\mu'}$ on the right, thus getting
$$
 x_\nu x_\mu x_{\mu'}= \ell_{\nu \mu}  x_\mu x_\nu x_{\mu'}
 + p_{\nu\mu}  x_{\mu'} x_{\nu'}x_{\mu'} \, . 
$$
By summing these two equalities, and using the antisymmetry of $p_{\mu\nu}$ we obtain
$$
x_{\mu'} x_\mu x_\nu + x_\nu x_\mu x_{\mu'} 
= \ell_{\mu \nu} \left( x_{\mu'} x_\nu x_\mu +   x_\mu x_\nu x_{\mu'} \right)
$$
or equivalently, by the simultaneous exchange $\mu \leftrightarrow \mu'$ and $\nu \leftrightarrow \nu'$,   equation 
\eqref{xx}.
\end{proof}

\begin{prop}\label{prop:R} 
Let $\ell_{\mu \nu}$ and $p_{\mu \nu}$ be parameters as in Definition~\ref{def:lp}. 
With the notation of the table~\eqref{tabella}, suppose that for each $\nu \in \int$ the parameters satisfy the relation 
\beq\label{condition-c}
\cc_\mn p_{\mn\nu} + \cc_\mt p_{\mt\nu} + \cc_\nt\ell_{\mt \nu} p_{\nt\nu} = 0 
\eeq
for some $\cc_\sigma \in \k[\ell_{\mu \nu}, p_{\mu \nu}]$, $\sigma \in \int$.
Then,
the 
element 
$
R_c:=\sum_{\mu=0}^3 \cc_\mu x_\mu^2
$ 
belongs to the center of  the algebra   $\alp$. 
\end{prop}
\begin{proof} 
For each $\nu$ we need to show that 
$$
\sum_\mu \cc_\mu x^2_\mu x_\nu - x_\nu \sum_\mu \cc_\mu x^2_\mu = 0 .
$$
Let us start with \eqref{comm-rel-x} for fixed indices $\mu ,\nu$. Multiplying 
it from the left by $x_\mu$ we get
$$
x_\mu^2 x_\nu = \ell_{\mu \nu} x_\mu x_\nu x_\mu + p_{\mu\nu} x_\mu x_{\nu'} x_{\mu'} \, . 
$$ 
Then exchange $\mu \leftrightarrow \nu$ in \eqref{comm-rel-x} and multiply the result 
by $x_\mu$ on the right to obtain:
$$
 x_\nu x^2_\mu= \ell_{\nu \mu}  x_\mu x_\nu x_\mu+ p_{\nu\mu}  x_{\mu'} x_{\nu'}x_\mu \, .
$$
When comparing these two expressions we have: 
 $$
\cc_\mu (x_\mu^2 x_\nu - x_\nu x^2_\mu) =  \cc_\mu p_{\mu\nu} \left(x_\mu x_{\nu'} x_{\mu'}
+ x_{\mu'} x_{\nu'}x_\mu \right) .
$$
Thus for each index $\nu$, by using Lemma~\ref{lem1} on the possible values of the indices
(and recalling that $p_{\mu \mu}=0$), we arrive at:
\begin{align*} 
\sum_\mu \cc_\mu (x^2_\mu x_\nu - x_\nu  x^2_\mu) 
& =  \sum_{\mu\neq \nu}  \cc_\mu p_{\mu\nu} \left(x_\mu x_{\nu'} x_{\mu'}
+ x_{\mu'} x_{\nu'}x_\mu \right)
\\
& = \cc_\mn p_{\mn\nu} \left(x_\mn  x_\nt  x_\mt + x_\mt x_\nt x_\mn  \right)
+ \cc_\mt p_{\mt\nu} \left(x_\mt x_\nt  x_\mn  + x_\mn  x_\nt x_\mt \right) 
\\ &  
\qquad 
+ \cc_\nt p_{\nt\nu} \left(x_\nt  x_\mn  x_\mt + x_\mt x_\mn x_\nt  \right)
\\
& =
(\cc_\mn p_{\mn\nu} + \cc_\mt p_{\mt\nu} )\left(x_\mn  x_\nt  x_\mt + x_\mt x_\nt x_\mn  \right)
\\
&
\qquad +\cc_\nt p_{\nt\nu} \left(x_\nt  x_\mn  x_\mt + x_\mt x_\mn x_\nt  \right) .
\end{align*}
Formula \eqref{xx} for $\nu=\nu$ and $\mu=\mt$ leads to:
$$
x_\mt x_\mn x_\nt   + x_\nt  x_\mn  x_\mt
=\ell_{\mt \nu} \left(x_\mt x_\nt x_\mn  +x_\mn  x_\nt  x_\mt \right)
$$
having used that if $(\mu, \nu)=(\mt,\nu)$, then $(\mu',\nu')=(\mn, \nt)$, as from relations \eqref{primes}.
Hence 
$$ 
\sum_\mu \cc_\mu (x^2_\mu x_\nu - x_\nu  x^2_\mu)=
(\cc_\mn p_{\mn\nu} + \cc_\mt p_{\mt\nu} + \cc_\nt \ell_{\mt \nu} p_{\nt\nu}  )\left(x_\mn  x_\nt  x_\mt + x_\mt x_\nt x_\mn  \right) =0
$$
due to the hypothesis \eqref{condition-c} on the parameters.
\end{proof} 

Notice that for fixed parameters $\ell_{\mu \nu}$ and $p_{\mu \nu}$, there might be different coefficients $\cc_\sigma$ for which \eqref{condition-c} is satisfied. This is the case for instance for Sklyanin algebras (and Connes--Dubois-Violette planes), as shown in Lemma~\ref{lem:condlp} below. In particular, as a direct consequence of the above Proposition, we have 
\begin{cor}\label{cor-central}
Let  the parameters $\ell_{\mu \nu}$ and $p_{\mu \nu}$ satisfy the relation 
\beq\label{condition}
p_{\mn\nu} + p_{\mt\nu} + \ell_{\mt \nu} p_{\nt\nu} = 0 \, .
\eeq
Then, the quadratic element 
$
R:=\sum_{\mu=0}^3 x_\mu^2
$ 
belongs to the center of  the algebra   $\alp$. 
\\
Suppose in addition there exist coefficients  $\cc_\sigma \in \k[\ell_{\mu \nu}, p_{\mu \nu}]$, for 
$\sigma \in \int$ such that  
\beq\label{condition-p}
(\cc_\mn -\cc_\nt)p_{\mn \nu}+(\cc_\mt -\cc_\nt)p_{\mt \nu}=0 \; , \quad \forall \nu \in \int.
\eeq
Then, the quadratic element 
$
R_c:=\sum_{\mu=0}^3 \cc_\mu x_\mu^2
$ 
belongs to the center of  the algebra   $\alp$. 
\end{cor}

\begin{defi}
Let $\ell_{\mu \nu}$ and $p_{\mu \nu}$ as in Definition~\ref{def:lp} be constrained by the relation \eqref{condition},
so that the element $R:=\sum_{\mu=0}^3 x_\mu^2$ is central. We denote by \beq\label{slp}
\aslp:= \alp / \langle R-1 \rangle 
\eeq
the quotient of the algebra $\alp$ by the two-sided ideal generated by $R-1$. 
\end{defi}
We refer to $\slp $ as a quantum three-sphere with coordinate algebra $\aslp$.
In the `classical limit', where for each $\mu, \nu$ we take $\ell_{\mu \nu} = 1$ 
and $p_{\mu \nu}=0$, the algebra $\aslp$ reduces to the algebra of polynomial functions on a three-sphere. 

Using the result in \eqref{ells}, conditions \eqref{condition} read, 
with $\ell_{30} = \ell_{21} = \ell_{12}= \ell_{03}$, 
\begin{align}\label{cond-expl}
& p_{20} + p_{30} + \ell_{03} ~p_{10 }=0 \quad  ;  & p_{31} + p_{21} + \ell_{03} ~p_{01 }=0  \nn
\\
& p_{02} + p_{12} + \ell_{03}~ p_{ 32} =0 \quad ;   & p_{13} + p_{03} + \ell_{03} ~p_{23 } = 0 \, .
\end{align}
Only three of these conditions are necessary, the fourth one  
$p_{13} + p_{03} + \ell_{03} p_{23 } =0$ (say), follows from the other three identities by simple substitutions.
 
\bigskip
Conditions \eqref{condition} are verified  for parameters $\ell$'s and $p$'s of the quantum planes $\Au$ of \cite{cdv} that we have described in \S\ref{sec:cdv}:
\begin{lem}
Let the parameters $\ell_{\mu \nu}$ and $p_{\mu \nu}$ be as in \S\ref{sec:cdv} for the four-planes of \cite{cdv}. 
Then condition \eqref{condition} is satisfied for each $\nu \in \int$.
\end{lem}
\begin{proof}
First observe that for each $\nu$ fixed, condition
$p_{\mn\nu} + p_{\mt\nu} + \ell_{\mt \nu} p_{\nt\nu}  =0$ is equivalent to 
$\lambda_\mn  q_{\mn\nu} + \lambda_\mt q_{\mt\nu} + \lambda_\nt  \ell_{\mt \nu} q_{\nt\nu}  =0$.
The proof is then by explicit computation for each of the three cases $\nu=0,1,2$. 
For instance, for $\nu=0$, using \eqref{def:l-p} and \eqref{ab}, 
\begin{align*}
p_{20} + p_{30} + \ell_{30} p_{10 }=0  
&\iff \lambda_2 q_{20} + \lambda_3 q_{30} + \lambda_1 \ell_{30} q_{10 } =0
\\
&\iff \lambda_2 a_{03} c_{20} + \lambda_3 a_{01} c_{30} + \lambda_1 b_{03} c_{10 } =0
\end{align*}
and this latter can be  proved with some algebra after substituting  \eqref{abc}. 
\end{proof}
For the sub-family in Example~\ref{ex:cl} it is even easier to see that \eqref{condition} is verified.

By  Proposition \ref{prop:R}, the above lemma gives that the element $R=\sum_{\mu=0}^3 x_\mu^2$ is central in the algebra $\Au$.  The corresponding quantum three-spheres $\slp$  are the noncommutative spherical manifolds $\Su$ introduced in \cite{cdv}. The centrality of the element $R$ was there deduced directly from the relations 
\eqref{rel-z}.

The centrality of the quadratic element $Q= -(\mathrm{x}_0)^2 + (\mathrm{x}_1)^2 + (\mathrm{x}_2)^2 + (\mathrm{x}_3)^2$ for the Sklyanin algebras in \S\ref{sec:skly} was originally mentioned in \cite[Thm.2]{Skly}, (c.f. also \cite[page 276]{smith}). In our setting, condition \eqref{cor-central} is verified for parameters $\ell_{\mu \nu}$ and $p_{\mu \nu} $ as in \eqref{lp-skly-C} characterizing the Sklyanin algebras over 
$\IC$ described in \S\ref{sec:skly-C}. Indeed, by direct computation one shows that:

\begin{lem}
The parameters $\ell_{\mu \nu}$ and $p_{\mu \nu}$  in \eqref{lp-skly-C}  satisfy condition \eqref{condition} and thus $\sum_\mu x_\mu^2$ is central in the corresponding algebra $\alp$. In particular for each $\nu=0,1,2,3$, condition \eqref{condition} is equivalent to
the Sklyanin condition  $\ell_{02}=\ell_{01}\ell_{03}$ (or the equivalent one in \eqref{cond-C}). 
\end{lem}

\begin{rem}
It is worth noticing that the parameters $\ell_{\mu \nu}$ and $p_{\mu \nu}$ for the Sklyanin algebras as in \eqref{lp-skly} do not satisfy condition \eqref{condition}. Indeed the above lemma shows that in terms of the generators $\mathrm{x}_\mu$ of \S\ref{sec:skly}, the central element is rather
$-\mathrm{x}_0^2 + \mathrm{x}_1^2 + \mathrm{x}_2^2 +\mathrm{x}_3^2$.     
\end{rem}
Finally, as shown in \cite[Thm.2]{Skly} for the Sklyanin algebras there is a second central element. In parallel with this we also have the following result:
\begin{lem}\label{lem:condlp}
Let $\alp$ be the  algebra defined by the parameters $\ell_{\mu \nu}$ and $p_{\mu \nu}$  in \eqref{lp-skly-C}.
Set 
$$
\cc_0:=0 \quad; \quad \cc_1:=\ell_{03}(1+\ell_{01}) \quad; \quad  \cc_2:=(1+\ell_{02}) 
\quad; \quad \cc_3:=(1+\ell_{03}) \, .
$$ 
Then condition \eqref{condition-p} is satisfied for each index $\nu=0,1,2,3$ and thus the element
$$ \ell_{03}(1+\ell_{01}) x_1^2 + (1+\ell_{02})x_2^2+ (1+\ell_{03})x_3^2 $$
belongs to the center of the algebra $\alp$, in addition to the element 
$\sum_{\mu=0}^3 x_\mu^2$ .
\end{lem}
\begin{proof}
The proof reduces to an explicit computation for each index $\nu=0,1,2,3$: 
\begin{align*}
\nu=0 &: \quad (1-\ell_{03})p_{20} + (1-\ell_{02})p_{30}
= -i p_{03}p_{20} +i p_{02}p_{30}=0
\\
\nu=1 &: \quad (1+\ell_{03})p_{31} + (1+\ell_{02})p_{21}
= -i p_{12}p_{31} +i p_{13}p_{21}=0
\\
\nu=2 &: \quad -(1+\ell_{03})p_{02} - (1-\ell_{02})p_{12}
= i p_{12}p_{02} -i p_{02}p_{12}=0
\\
\nu=3 &: \quad -(1-\ell_{03})p_{13} - (1+\ell_{02})p_{03}
= i p_{03}p_{13} -i p_{13}p_{03}=0 \, ,
\end{align*}
showing that condition \eqref{condition-p} is satisfied.
\end{proof}
 
The differential calculi $\omlp=\omlp(\alp)$ constructed in \S \ref{sec:calculus} can be restricted to the noncommutative spheres 
$\slp$. The differential graded algebra $\Omega(\slp)$ of the calculus on the sphere $\slp$ is  defined to be
the quotient of $\omlp$ by the differential ideal generated by $R-1$, equipped with the induced differential $\dd$. Explicitly, at first order, 
$$
\Omega^0(\slp):=\aslp=\omlp^0 / \langle \sum_\mu x_\mu^2-1 \rangle 
\, , \qquad
 \Omega^1(\slp):=\omlp^1 / \langle \sum_\mu x_\mu \dx_\mu \rangle \, ,
$$
with $x_\mu \dx_\mu = \dx_\mu x_\mu $ for each $\mu$, as from the relations \eqref{one-forms}.

\section{Symmetries of $\alp$ }\label{sec:mlp}

In Definition~\ref{def:lp} we have introduced a class of quadratic algebras $\alp$, associated to  parameters $\ell_{\mu \nu}$ and $p_{\mu \nu}$ satisfying some suitable conditions, and in \S\ref{sec:calculus} we have constructed their exterior algebras $(\omlp,\dd)$. 
In the present section we proceed with the study of the quadratic differential algebras $\alp$ and  construct transformation bialgebras for them.

\subsection{The symmetry bialgebra $\mlp$}
We aim at the construction of a bialgebra $\mlp$ and an algebra map
$\delta: \alp \to \mlp \ot \alp$ which defines a coaction of $\mlp$ on the quadratic differential algebra $\alp$. Recall that $\alp$ is the $\IN$-graded algebra, finitely presented with degree one generators $x_\mu$, $\mu=0,1,2,3$ and finite homogeneous relations \eqref{comm-rel-x} of degree 2: 
$$
\alp= \bigoplus_{n \in \mathbb{N}}
\alp^n = \k \langle x_0, x_1,x_2,x_3 \rangle / I \, , 
$$
where $I$ is the ideal of relations generated by  all quadratic relations 
\eqref{comm-rel-x}, and $\alp^0=\k$.
We determine the bialgebra $\mlp$ and the algebra map $\delta: \alp \to \mlp \ot \alp$ by  requiring:
\begin{enumerate}[(I)]
\item\label{cond1} \quad
$\delta$ is an algebra map that 
preserves the $\mathbb{N}$-grading: $\delta: \alp^n \to \mlp \ot \alp^n$ ,
\item\label{cond3} \quad 
$\delta$ extends to a coaction on the differential algebra $(\omlp,\dd)$ of $\alp$ by requiring 
\beq\label{dd}
\delta \circ \dd = (\id \ot \dd) \, \delta \; .
\eeq 
\end{enumerate}

Firstly, since $\delta$ should be an algebra map, it is determined by its value on the algebra generators of $\alp$. And since it has to respect the $\mathbb{N}$-grading of $\alp$, that is condition \eqref{cond1}, it has  to be of the form
\beq\label{coact}
\delta: \alp \to \mlp \ot \alp ~, \qquad x_\mu \mapsto \sum_\nu \A_{\mu \nu} \ot x_\nu
\eeq
for elements $\A_{\mu \nu}$  in 
$\mlp$, $\mu,\nu=0,1,2,3$. 
We wish $\mlp$ to be an algebra which is finitely generated by these elements $\A_{\mu \nu}$ and determine the minimal algebra properties 
they have to satisfy in order 
for $\delta: x_\mu \mapsto \sum_\nu \A_{\mu \nu} \ot x_\nu$
to preserve the ideal of relations: $\delta(I)=0$. For this, 
fix indices $\mu,\nu \in \int$ and apply the algebra map $\delta$ to the corresponding defining relation \eqref{comm-rel-x} to compute:
\begin{align}\label{deltaI}
\delta( x_\mu x_\nu - \ell_{\mu \nu} x_\nu  x_\mu & -  p_{\mu\nu}  x_{\nu'} x_{\mu'})
= 
\delta( x_\mu) \delta (x_\nu )- \ell_{\mu \nu} \delta (x_\nu) \delta ( x_\mu )-  p_{\mu\nu} \delta (x_{\nu'}) \delta (x_{\mu'}) \nn
\\ =& 
\sum_{\alpha, \beta} 
\left( \A_{\mu \alpha} \A_{\nu \beta} - 
\ell_{\mu \nu} \A_{\nu \alpha}\A_{\mu \beta} 
-  p_{\mu\nu} \A_{\nu' \alpha}\A_{\mu' \beta} 
  \right) \ot x_\alpha x_\beta \nn
\\
=& \sum_{\alpha} 
\left( \A_{\mu \alpha} \A_{\nu \alpha} - 
\ell_{\mu \nu} \A_{\nu \alpha}\A_{\mu \alpha} 
-  p_{\mu\nu} \A_{\nu' \alpha}\A_{\mu' \alpha} 
  \right) \ot x_\alpha^2 \nn
\\
& 
 \, + \sum_{\alpha < \beta} 
\left( \A_{\mu \alpha} \A_{\nu \beta} - 
\ell_{\mu \nu} \A_{\nu \alpha}\A_{\mu \beta} 
-  p_{\mu\nu} \A_{\nu' \alpha}\A_{\mu' \beta} 
  \right) \ot x_\alpha x_\beta \nn
\\
&
 \, +
\sum_{\alpha < \beta} 
\left( \A_{\mu \beta} \A_{\nu \alpha} - 
\ell_{\mu \nu} \A_{\nu \beta}\A_{\mu \alpha} 
-  p_{\mu\nu} \A_{\nu' \beta}\A_{\mu' \alpha} 
  \right)  \ot ~\ell_{\alpha \beta} x_\alpha x_\beta \nn
\\
&
 \, + 
\sum_{\alpha < \beta} 
\left( \A_{\mu \beta'} \A_{\nu \alpha'} - 
\ell_{\mu \nu} \A_{\nu \beta'}\A_{\mu \alpha'} 
-  p_{\mu\nu} \A_{\nu' \beta'}\A_{\mu' \alpha'} 
  \right) \ot p_{\beta' \alpha'} x_\alpha x_\beta  \, ,
\end{align}
where 
the last two summands have been obtained by expressing $x_\alpha x_\beta$ for $\alpha >\beta$  in terms of $x_\beta x_\alpha$ and $x_{\beta'} x_{\alpha'}$ via \eqref{comm-rel-x} and then renaming the indices.

As observed after \eqref{exre} the monomials $x_\alpha x_\beta$, $\alpha \leq \beta$ 
form a vector space basis for $\alp^2$. Thus, the map $\delta$ preserves the algebra relations, that is the right hand side of the expression in \eqref{deltaI} vanishes 
if and only if the coefficients of $x_\alpha^2$ 
and of $x_\alpha x_\beta$ for $\alpha \leq \beta$ vanish, that is the elements $\A_{\mu \nu}$  satisfy
\beq\label{id1}
\A_{\mu \alpha} \A_{\nu \alpha} =
\ell_{\mu \nu} \A_{\nu \alpha}\A_{\mu \alpha} 
+  p_{\mu\nu} \A_{\nu' \alpha}\A_{\mu' \alpha} 
\eeq
for all $\alpha$, $\mu$, $\nu$, and for all $\alpha<\beta$, and for all $\mu$, $\nu$:
\begin{align}\label{eq1}
0 = &\A_{\mu \alpha} \A_{\nu \beta} - 
\ell_{\mu \nu} \A_{\nu \alpha}\A_{\mu \beta} 
-  p_{\mu\nu} \A_{\nu' \alpha}\A_{\mu' \beta} 
+ \ell_{\alpha \beta} \A_{\mu \beta} \A_{\nu \alpha} - 
\ell_{\alpha \beta}\ell_{\mu \nu} \A_{\nu \beta}\A_{\mu \alpha} 
\nn
\\ & 
-  \ell_{\alpha \beta} p_{\mu\nu} \A_{\nu' \beta}\A_{\mu' \alpha}  
+  p_{\beta' \alpha'}  \A_{\mu \beta'} \A_{\nu \alpha'} 
 - 
p_{\beta' \alpha'}  \ell_{\mu \nu} \A_{\nu \beta'}\A_{\mu \alpha'} 
-  p_{\beta' \alpha'}  p_{\mu\nu} \A_{\nu' \beta'}\A_{\mu' \alpha'} \, .
\end{align} 
In fact, we could equally have expressed \eqref{deltaI} in terms of $x_\alpha x_\beta$ for $\alpha>\beta$ and the expression would have been the same. Thus, condition \eqref{eq1} should hold for all $\alpha,\beta$, in particular also for $\alpha=\beta$ since in this case it reduces to \eqref{id1}.

We next require that $\delta$ satisfies condition \eqref{cond3}. Property \eqref{dd} yields 
$\delta$ on one-forms:
$$
\delta: \dx_\mu \mapsto \sum_\nu \A_{\mu \nu} \ot \dx_\nu \; \mbox{ for all } \mu=0,1,2,3~ .
$$  
Extending $\delta$ to an algebra map on $\omlp$ which now preserves conditions 
\eqref{two-forms} on basis one-forms, by proceeding as before, we compute
\begin{align*}
0&=\delta( \dd x_\mu) \delta (\dd x_\nu ) +\ell_{\mu \nu} \delta (\dd x_\nu) \delta ( \dd x_\mu )+  p_{\mu\nu} \delta (\dd x_{\nu'}) \delta (\dd x_{\mu'})
\\
&= 
 \sum_{\alpha < \beta} 
\left( \A_{\mu \alpha} \A_{\nu \beta} 
+ 
\ell_{\mu \nu} \A_{\nu \alpha}\A_{\mu \beta} 
+  p_{\mu\nu} \A_{\nu' \alpha}\A_{\mu' \beta} 
  \right) \ot \dd x_\alpha \dd x_\beta 
\\
&
 \,\, -
\sum_{\alpha < \beta} 
\left( \A_{\mu \beta} \A_{\nu \alpha} + 
\ell_{\mu \nu} \A_{\nu \beta}\A_{\mu \alpha} 
+  p_{\mu\nu} \A_{\nu' \beta}\A_{\mu' \alpha} 
  \right)  \ot ~\ell_{\alpha \beta} \dd x_\alpha \dd x_\beta 
\\
&  
 \,\, -
 \sum_{\alpha < \beta} 
\left( \A_{\mu \beta'} \A_{\nu \alpha'} + 
\ell_{\mu \nu} \A_{\nu \beta'}\A_{\mu \alpha'} 
+  p_{\mu\nu} \A_{\nu' \beta'}\A_{\mu' \alpha'} 
  \right) \ot p_{\beta' \alpha'} \dd x_\alpha \dd x_\beta \, .
\end{align*}
That is, for each $\mu,\nu,\alpha \neq \beta$ 
(again with either the condition $\alpha < \beta$ or $\alpha > \beta$), we get
\begin{align}\label{eq2}
0 = &\A_{\mu \alpha} \A_{\nu \beta} + 
\ell_{\mu \nu} \A_{\nu \alpha}\A_{\mu \beta} 
+  p_{\mu\nu} \A_{\nu' \alpha}\A_{\mu' \beta} 
- \ell_{\alpha \beta} \A_{\mu \beta} \A_{\nu \alpha}  
- \ell_{\alpha \beta}\ell_{\mu \nu} \A_{\nu \beta}\A_{\mu \alpha} 
\nn
\\ &
-  \ell_{\alpha \beta} p_{\mu\nu} \A_{\nu' \beta}\A_{\mu' \alpha}  
-  p_{\beta' \alpha'}  \A_{\mu \beta'} \A_{\nu \alpha'} 
-p_{\beta' \alpha'}  \ell_{\mu \nu} \A_{\nu \beta'}\A_{\mu \alpha'} 
- p_{\beta' \alpha'}  p_{\mu\nu} \A_{\nu' \beta'}\A_{\mu' \alpha'} .
\end{align} 
Comparing \eqref{eq2} with equation \eqref{eq1} obtained for the zero-forms part, 
we have that in order for these equations to be satisfied the elements $\A_{\mu \nu}$ have to fulfil the conditions 
\begin{align} \label{cond-equiv}
 \A_{\mu \alpha} \A_{\nu \beta} 
 - \ell_{\alpha \beta}\ell_{\mu \nu} \A_{\nu \beta}\A_{\mu \alpha} 
& -  \ell_{\alpha \beta} p_{\mu\nu} \A_{\nu' \beta}\A_{\mu' \alpha} 
\nn \\ & \qquad - p_{\beta' \alpha'}  \ell_{\mu \nu} \A_{\nu \beta'}\A_{\mu \alpha'} 
-  p_{\beta' \alpha'}  p_{\mu\nu} \A_{\nu' \beta'}\A_{\mu' \alpha'} = 0 \, ,
\nn \\
~ \nn \\
 - \ell_{\mu \nu} \A_{\nu \alpha}\A_{\mu \beta} 
-  p_{\mu\nu} \A_{\nu' \alpha}\A_{\mu' \beta} 
& + \ell_{\alpha \beta} \A_{\mu \beta} \A_{\nu \alpha} 
+  p_{\beta' \alpha'}  \A_{\mu \beta'} \A_{\nu \alpha'}  =0 \, . 
\end{align}

In Proposition \ref{prop:comm-rel-equiv} below, we will show that these two conditions are equivalent, that is, it is enough that one of them is satisfied for the other to be satisfied as well. We are hence led to give the following definition.

\begin{defi}\label{def:mlp}
Let $\mlp$ be the associative $\mathbb{N}$-graded algebra  generated in degree-one by elements $\A_{\mu \alpha}$, $\mu,\alpha=0,1,2,3$ and defining relations in degree-two
\beq\label{comm-rel-A}
\A_{\mu \alpha}\A_{\nu \beta}= \ell_{\mu \nu} \ell_{\beta \alpha } \A_{\nu \beta}\A_{\mu \alpha} +
p_{\mu \nu} \ell_{\beta \alpha } \A_{\nu' \beta}\A_{\mu' \alpha} +
\ell_{\mu \nu} p_{\beta' \alpha'} \A_{\nu \beta'}\A_{\mu \alpha'} +
p_{\mu \nu} p_{\beta' \alpha'} \A_{\nu' \beta'}\A_{\mu' \alpha'} 
\eeq
for all $\mu,\nu,\alpha, \beta \in \int$, and where $(\mu',\nu')=(\mu,\nu)'$ and
$(\alpha',\beta')=(\alpha,\beta)'$. 
\end{defi}
In the spirit of quantized  algebras of coordinate functions on matrix groups, we can think at $\mlp$ as the algebra generated by the entries (coordinate functions) of  a  matrix 
$M=(\A_{\mu \nu}, \, \mu, \nu =0,1,2,3)$ subject to the commutation relations \eqref{comm-rel-A}.

The relations \eqref{comm-rel-A} do not generate additional relations in degree-two in the sense that 
by  applying them twice we return to the monomial we started from. 
Using the defining conditions on the parameters: 
$\ell_{\mu \nu}= \ell_{\nu \mu}=\ell_{\mu' \nu'}= \ell_{\nu' \mu'}$ and $p_{\mu \nu}= -p _{\nu \mu}$ we compute
\begin{align*}
&\A_{\mu \alpha}\A_{\nu \beta} =
\\
&=  \ell_{\mu \nu} \ell_{\beta \alpha } 
\big(
\ell_{\mu \nu} \ell_{\beta \alpha } \A_{\mu \alpha} \A_{\nu \beta}  
+
p_{\nu \mu} \ell_{\beta \alpha } \A_{\mu' \alpha}  \A_{\nu' \beta}
 + \ell_{\mu \nu} p_{\alpha'\beta' }\A_{\mu \alpha'} \A_{\nu \beta'} 
+
p_{\nu \mu} p_{\alpha' \beta' } \A_{\mu' \alpha'} \A_{\nu' \beta'} 
\big)
\\[3pt]&
+ p_{\mu \nu} \ell_{\beta \alpha } 
\big(
\ell_{\mu \nu} \ell_{\beta \alpha } \A_{\mu' \alpha} \A_{\nu' \beta}  +
p_{\nu' \mu'} \ell_{\beta \alpha } \A_{\mu \alpha}  \A_{\nu \beta}  
 + \ell_{\mu \nu} p_{\alpha'\beta' }\A_{\mu' \alpha'} \A_{\nu' \beta'} +
p_{\nu' \mu'} p_{\alpha' \beta' } \A_{\mu \alpha'} \A_{\nu \beta'} 
\big)
\\[3pt]&
+ \ell_{\mu \nu} p_{\beta' \alpha' } 
\big(
\ell_{\mu \nu} \ell_{\beta \alpha } \A_{\mu \alpha'} \A_{\nu \beta'}  +
p_{\nu \mu} \ell_{\beta \alpha} \A_{\mu' \alpha'}  \A_{\nu' \beta'} 
 +
\ell_{\mu \nu} p_{\alpha\beta }\A_{\mu \alpha} \A_{\nu \beta} +
p_{\nu \mu} p_{\alpha \beta } \A_{\mu' \alpha} \A_{\nu' \beta} 
\big)
\\[3pt]&
+ p_{\mu \nu} p_{\beta' \alpha' } 
\big(
\ell_{\mu \nu} \ell_{\beta \alpha } \A_{\mu' \alpha'} \A_{\nu' \beta'}  +
p_{\nu' \mu'} \ell_{\beta \alpha } \A_{\mu \alpha'}  \A_{\nu \beta'}
 +
\ell_{\mu \nu} p_{\alpha\beta }\A_{\mu' \alpha} \A_{\nu' \beta} +
p_{\nu' \mu'} p_{\alpha \beta } \A_{\mu \alpha} \A_{\nu \beta} 
\big)
\\[3pt]& 
=  
\big( \ell_{\mu \nu}^2 \ell_{\beta \alpha }^2+ p_{\mu \nu} \ell_{\beta \alpha }^2  p_{\nu' \mu'}  +  \ell_{\mu \nu}^2 p_{\beta' \alpha' } p_{\alpha\beta } +
p_{\mu \nu} p_{\beta' \alpha' } p_{\nu' \mu'} p_{\alpha \beta }
\big) \A_{\mu \alpha}  \A_{\nu \beta} 
\\[3pt]&
+ \big( \ell_{\mu \nu} \ell_{\beta \alpha }^2 p_{\nu \mu} +
\ell_{\beta \alpha }^2 \ell_{\mu \nu} p_{\mu \nu} 
 + \ell_{\mu \nu} p_{\beta' \alpha' }  p_{\nu \mu} p_{\alpha \beta } 
+ p_{\mu \nu} p_{\beta' \alpha' }\ell_{\mu' \nu'} p_{\alpha\beta }
\big) \A_{\mu' \alpha} \A_{\nu' \beta} 
\\[3pt]&
+ \big(\ell_{\mu \nu}^2 \ell_{\beta \alpha } p_{\alpha'\beta' }+
p_{\mu \nu} \ell_{\beta \alpha } p_{\nu' \mu'} p_{\alpha' \beta' }
 + \ell_{\mu \nu}^2 p_{\beta' \alpha' } \ell_{\beta \alpha } 
+p_{\mu \nu} p_{\beta' \alpha' } p_{\nu' \mu'} \ell_{\beta \alpha}
\big)\A_{\mu \alpha'} \A_{\nu \beta'} 
\\[3pt]&
+ \big(\ell_{\mu \nu} \ell_{\beta \alpha }  p_{\nu \mu} p_{\alpha' \beta' }+
p_{\mu \nu} \ell_{\beta \alpha } \ell_{\mu \nu} p_{\alpha'\beta' } 
 + \ell_{\mu \nu} p_{\beta' \alpha' }p_{\nu \mu} \ell_{\beta \alpha}
+ p_{\mu \nu} p_{\beta' \alpha' } \ell_{\mu \nu} \ell_{\beta \alpha }
\big)\A_{\mu' \alpha'} \A_{\nu' \beta'} 
\\[3pt]&
=   
\big( \ell_{\mu \nu}^2 \ell_{\beta \alpha }^2+ p_{\mu \nu} \ell_{\beta \alpha }^2  p_{\nu' \mu'}
 +  \ell_{\mu \nu}^2 p_{\beta' \alpha' } p_{\alpha\beta } +
p_{\mu \nu} p_{\beta' \alpha' } p_{\nu' \mu'} p_{\alpha \beta }
\big) \A_{\mu \alpha}  \A_{\nu \beta}
\end{align*}
where in the last step we have simply used the properties of the $\ell$'s and the $p$'s 
to conclude that all coefficients vanish, but for the first one. Finally, by using\eqref{lp} in 
Definition \ref{def:lp}, the coefficient in parenthesis is worked out to be just
\begin{align*}
\ell_{\mu \nu}^2 \ell_{\beta \alpha }^2+ p_{\mu \nu} \ell_{\beta \alpha }^2  p_{\nu' \mu'}+ & \ell_{\mu \nu}^2 p_{\beta' \alpha' } p_{\alpha\beta } +
p_{\mu \nu} p_{\beta' \alpha' } p_{\nu' \mu'} p_{\alpha \beta }
 \\& = 
\ell_{\beta \alpha }^2 \left(\ell_{\mu \nu}^2 + p_{\mu \nu}   p_{\nu' \mu'} \right)+ 
p_{\beta' \alpha' } p_{\alpha\beta }  \left( \ell_{\mu \nu}^2 +
p_{\mu \nu}  p_{\nu' \mu'} \right)
\\
& =  \ell_{\beta \alpha }^2  + p_{\beta' \alpha' } p_{\alpha\beta }  =1 \, .
\end{align*}
This gives again $\A_{\mu \alpha}  \A_{\nu \beta}$ for the right hand side. 

We next show that the two conditions in \eqref{cond-equiv} are equivalent.
\begin{prop}\label{prop:comm-rel-equiv}
The generators $\A_{\mu \alpha}$ of $\mlp$ satisfy conditions \eqref{comm-rel-A}
if and only if, for all $\mu,\nu,\alpha, \beta \in \int$, with $(\mu',\nu')=(\mu,\nu)'$ and
$(\alpha',\beta')=(\alpha,\beta)'$, they satisfy
\beq\label{rel-A}
\ell_{\alpha \beta} \A_{\mu \alpha} \A_{\nu \beta} 
+
 p_{\alpha' \beta'}  \A_{\mu \alpha'} \A_{\nu \beta'} =
\ell_{\mu \nu} \A_{\nu \beta}\A_{\mu \alpha} 
+  p_{\mu\nu} \A_{\nu' \beta}\A_{\mu' \alpha} \, .
\eeq 
\end{prop}
\begin{proof}
Firstly we show that equations \eqref{comm-rel-A} imply that
$$
 - 
\ell_{\mu \nu} \A_{\nu \alpha}\A_{\mu \beta} 
-  p_{\mu\nu} \A_{\nu' \alpha}\A_{\mu' \beta} 
+ \ell_{\alpha \beta} \A_{\mu \beta} \A_{\nu \alpha} 
+
 p_{\beta' \alpha'}  \A_{\mu \beta'} \A_{\nu \alpha'} 
$$ 
is zero for all $\mu,\nu,\alpha,\beta$. Indeed, by applying \eqref{comm-rel-A} to the first 
two addends of the above expression,  this is modified to
\begin{align*}
& - \ell_{\mu \nu} \big(  
\ell_{\mu \nu} \ell_{\alpha \beta}  \A_{\mu \beta}  \A_{\nu \alpha}
+
p_{\nu \mu} \ell_{\beta \alpha } \A_{\mu' \beta}\A_{\nu' \alpha} 
+ \ell_{\nu \mu} p_{\beta' \alpha'} \A_{\mu \beta'}\A_{\nu \alpha'} 
+
 p_{\nu \mu} p_{\beta' \alpha'} \A_{\mu' \beta'}\A_{\nu' \alpha'} \big)
\\ 
& - p_{\mu\nu}  \big(  
\ell_{\nu \mu} \ell_{\beta \alpha } \A_{\mu' \beta}\A_{\nu' \alpha} 
 +   p_{\nu' \mu'} \ell_{\beta \alpha } \A_{\mu \beta}\A_{\nu \alpha} 
+  \ell_{\nu \mu} p_{\beta' \alpha'} \A_{\mu' \beta'}\A_{\nu' \alpha'}
 +
 p_{\nu' \mu'} p_{\beta' \alpha'} \A_{\mu \beta'}\A_{\nu \alpha'} 
\big) 
\\ 
&
+  \ell_{\alpha \beta} \A_{\mu \beta} \A_{\nu \alpha} 
 + p_{\beta' \alpha'}  \A_{\mu \beta'} \A_{\nu \alpha'} 
\\
= & -
\big(
\ell_{\mu \nu}^2 \ell_{\alpha \beta}  +  p_{\mu\nu} 
p_{\nu' \mu'} \ell_{\beta \alpha } -
\ell_{\alpha \beta}
\big )\A_{\mu \beta}  \A_{\nu \alpha}  
 -
\big(  
\ell_{\nu \mu} p_{\nu \mu} \ell_{\beta \alpha } +  p_{\mu\nu} 
\ell_{\nu' \mu'} \ell_{\beta \alpha } \big) \A_{\mu' \beta}\A_{\nu' \alpha}
\\
&-
\big( 
\ell_{\nu \mu}^2 p_{\beta' \alpha'} +  p_{\mu\nu} 
p_{\nu' \mu'} p_{\beta' \alpha'}  -
 p_{\beta' \alpha'} \big) \A_{\mu \beta'}\A_{\nu \alpha'} 
  -
\big( 
\ell_{\nu \mu} p_{\nu \mu} p_{\beta' \alpha'} +  p_{\mu\nu} 
\ell_{\nu' \mu'} p_{\beta' \alpha'} \big) \A_{\mu' \beta'}\A_{\nu' \alpha'} 
\end{align*}
which vanishes since all coefficients are zero.

We are left to prove the converse: relations \eqref{rel-A} imply \eqref{comm-rel-A}, that is that
$$
\ell_{\mu \nu} \ell_{\beta \alpha } \A_{\nu \beta}\A_{\mu \alpha} +
p_{\mu \nu} \ell_{\beta \alpha } \A_{\nu' \beta}\A_{\mu' \alpha} +
\ell_{\mu \nu} p_{\beta' \alpha'} \A_{\nu \beta'}\A_{\mu \alpha'} +
p_{\mu \nu} p_{\beta' \alpha'} \A_{\nu' \beta'}\A_{\mu' \alpha'} = \A_{\mu \alpha}\A_{\nu \beta}.
$$
By using \eqref{rel-A}  twice we can rewrite the left hand side as
\begin{align*}
& ~ 
\ell_{\beta \alpha }  \big( \ell_{\mu \nu} \A_{\nu \beta}\A_{\mu \alpha} +
p_{\mu \nu}  \A_{\nu' \beta}\A_{\mu' \alpha} \big) +
p_{\beta' \alpha'}  \big(
\ell_{\mu \nu} \A_{\nu \beta'}\A_{\mu \alpha'} +
p_{\mu \nu} \A_{\nu' \beta'}\A_{\mu' \alpha'}  \big)
\\
  = & 
~\ell_{\beta \alpha }  \big(
\ell_{\beta \alpha} \A_{\mu \alpha} \A_{\nu \beta} 
+
 p_{\alpha' \beta'}  \A_{\mu \alpha'} \A_{\nu \beta'}
 \big) +
p_{\beta' \alpha'}  
\big(
\ell_{\beta \alpha} \A_{\mu \alpha'} \A_{\nu \beta'} 
+
 p_{\alpha \beta}  \A_{\mu \alpha} \A_{\nu \beta}
  \big)
\\
 = &  ~
\big(\ell_{\beta \alpha } \ell_{\beta \alpha} + p_{\beta' \alpha'}  p_{\alpha \beta}  \big) \A_{\mu \alpha} \A_{\nu \beta} 
\end{align*}
which is indeed equal to $
\A_{\mu \alpha} \A_{\nu \beta} $ because of condition \eqref{lp} in Definition~\ref{def:lp}. 
\end{proof}

Before to proceed and introduce a coalgebra structure for $\mlp$ let us observe the following.
By comparing \eqref{comm-rel-x} with \eqref{comm-rel-A}, we can immediately conclude that
\begin{lem}\label{lem:obs1}
For each $\alpha \in \int$ fixed, the map
\beq
\mlp \to \alp \; , \quad \A_{\mu \alpha} \mapsto x_\mu
\eeq
is an algebra isomorphism between the subalgebra of $\mlp$ generated by the
elements $\{\A_{\mu \alpha}\}_{\mu \in \int}$ in the $\alpha$-th column of the matrix $M$ 
and the algebra $\alp$.  
\end{lem}

Next we introduce coproduct and counit for the matrix algebra $\mlp$, compatible with its 
algebra structure.
\begin{prop} 
The linear maps defined on the generators by 
\begin{eqnarray} \label{cop-cou}
&  \Delta: \mlp \to \mlp \ot \mlp \quad ,  \quad \A_{\mu \nu} \mapsto \sum_{\alpha=0}^{3} \A_{\mu \alpha} \ot \A_{\alpha \nu}  \nn
\\
&  \varepsilon: \mlp \to \k \quad , \quad  \A_{\mu \nu} \mapsto \delta_{\mu \nu} ~1~,
\end{eqnarray}
and extended as algebra morphisms, endow  $\mlp$ with a bialgebra structure. 
\end{prop}
\begin{proof}
We need to show that the maps $\Delta$ and $\varepsilon$  preserve the commutation relations \eqref{comm-rel-A} and thus are well-defined algebra maps. As for the counit, it is well-defined if and only if
$$
 \delta_{\mu \alpha}\delta_{\nu \beta}= 
\ell_{\mu \nu} \ell_{\beta \alpha } \delta_{\nu \beta}\delta_{\mu \alpha} +
p_{\mu \nu} \ell_{\beta \alpha } \delta_{\nu' \beta}\delta_{\mu' \alpha} +
\ell_{\mu \nu} p_{\beta' \alpha'} \delta_{\nu \beta'}\delta_{\mu \alpha'} +
p_{\mu \nu} p_{\beta' \alpha'} \delta_{\nu' \beta'}\delta_{\mu' \alpha'} 
$$
that is, if and only if
$$
  \delta_{\mu \alpha}\delta_{\nu \beta}= 
\left( \ell_{\mu \nu} \ell_{\beta \alpha } + p_{\mu \nu} p_{\beta' \alpha'}  \right)  \delta_{\nu \beta}\delta_{\mu \alpha} +
\left(p_{\mu \nu} \ell_{\beta \alpha }  + \ell_{\mu \nu} p_{\beta' \alpha'}
\right)\delta_{\nu' \beta}\delta_{\mu' \alpha} \, .
$$  By using the conditions in Definition \eqref{def:lp} one directly verifies that this is the case:
$$
\left\{
\begin{array}{ll}
\ell_{\mu \nu} \ell_{\beta \alpha } + p_{\mu \nu} p_{\beta' \alpha'} = 1 & \mbox{ when } \nu = \beta, \mu= \alpha
\\
p_{\mu \nu} \ell_{\beta \alpha }  + \ell_{\mu \nu} p_{\beta' \alpha'} =0 
& \mbox{ when } \nu = \beta', \mu= \alpha'
\end{array}
\right . \, .
$$
For the coproduct we compute
\begin{align*}
&\Delta(\A_{\mu \alpha})\Delta(\A_{\nu \beta}) 
 - \ell_{\mu \nu} \ell_{\beta \alpha } \Delta(\A_{\nu \beta})\Delta(\A_{\mu \alpha}) +
p_{\mu \nu} \ell_{\beta \alpha } \Delta(\A_{\nu' \beta})\Delta(\A_{\mu' \alpha}) 
\\
 & \hspace{2.1cm} + \ell_{\mu \nu} p_{\beta' \alpha'} \Delta(\A_{\nu \beta'})\Delta(\A_{\mu \alpha'})
+ p_{\mu \nu} p_{\beta' \alpha'} \Delta(\A_{\nu' \beta'})\Delta(\A_{\mu' \alpha'}) 
\\
= &\sum_{\gamma,\tau} \big[ \A_{\mu \gamma} \A_{\nu \tau} \ot \A_{\gamma \alpha} \A_{\tau \beta}
 - \left( 
\ell_{\mu \nu} \ell_{\beta \alpha } \A_{\nu \tau} \A_{\mu \gamma} +
p_{\mu \nu} \ell_{\beta \alpha } \A_{\nu' \tau} \A_{\mu' \gamma}
\right) \ot  
\A_{\tau \beta} \A_{\gamma \alpha} +
\\
& -  \left( 
\ell_{\mu \nu} p_{\beta' \alpha'} \A_{\nu \tau} \A_{\mu \gamma} +
p_{\mu \nu} p_{\beta' \alpha'} \A_{\nu' \tau} \A_{\mu' \gamma} 
\right) \ot  \A_{\tau \beta'} \A_{\gamma \alpha'} \big] \, .
\end{align*}
We now use equations \eqref{comm-rel-A}  to rewrite the first factor of the tensor product  in the first addend and the second factors in the second and third addends:
\begin{align*}
\sum_{\gamma,\tau} \Big[ \Big( &\ell_{\mu \nu} \ell_{\tau \gamma } \A_{\nu \tau}\A_{\mu \gamma} +
p_{\mu \nu} \ell_{\tau \gamma } \A_{\nu' \tau}\A_{\mu' \gamma} +
 \ell_{\mu \nu} p_{\tau' \gamma'} \A_{\nu \tau'}\A_{\mu \gamma'} +
p_{\mu \nu} p_{\tau' \gamma'} \A_{\nu' \tau'}\A_{\mu' \gamma'} \Big)
\\
&  \hspace{12cm} \ot \A_{\gamma \alpha} \A_{\tau \beta} 
\\
-
 \big( &
\ell_{\mu \nu} \ell_{\beta \alpha } \A_{\nu \tau} \A_{\mu \gamma} +
p_{\mu \nu} \ell_{\beta \alpha } \A_{\nu' \tau} \A_{\mu' \gamma}
\big)  \ot  
\Big(
\ell_{\tau \gamma} \ell_{\alpha \beta } \A_{\gamma \alpha}\A_{\tau \beta} +
p_{\tau \gamma} \ell_{\alpha \beta } \A_{\gamma' \alpha}\A_{\tau' \beta}  
\\
&  \hspace{7cm} 
+ \ell_{\tau \gamma} p_{\alpha' \beta'} \A_{\gamma \alpha'}\A_{\tau \beta'}
+ p_{\tau \gamma} p_{\alpha' \beta'} \A_{\gamma' \alpha'}\A_{\tau' \beta'} 
\Big)
\\
- \big( &
\ell_{\mu \nu} p_{\beta' \alpha'} \A_{\nu \tau} \A_{\mu \gamma} +
p_{\mu \nu} p_{\beta' \alpha'} \A_{\nu' \tau} \A_{\mu' \gamma} 
\big) \ot 
\Big(
\ell_{\tau \gamma} \ell_{\alpha' \beta' } \A_{\gamma \alpha'}\A_{\tau \beta'} +
p_{\tau \gamma} \ell_{\alpha' \beta' } \A_{\gamma' \alpha'}\A_{\tau' \beta'} 
\\
&  \hspace{7.3cm} 
 +  \ell_{\tau \gamma} p_{\alpha \beta} \A_{\gamma \alpha}\A_{\tau \beta} +
p_{\tau \gamma} p_{\alpha \beta} \A_{\gamma' \alpha}\A_{\tau' \beta}
 \Big) \Big]
\end{align*}
thus obtaining
\begin{align*}   
\sum_{\gamma,\tau} & \Big\{
\big(\ell_{\mu \nu} \A_{\nu \tau} \A_{\mu \gamma} + p_{\mu \nu} \A_{\nu' \tau} \A_{\mu' \gamma} \big)  \ot  \Big[
\big( \ell_{\tau \gamma } - \ell_{\beta \alpha }^2 \ell_{\tau \gamma } -
p_{\beta' \alpha'}  \ell_{\tau \gamma} p_{\alpha \beta} 
\big)  \A_{\gamma \alpha}\A_{\tau \beta} 
\\ &
 \quad - \big(\ell_{\beta \alpha}^2 p_{\tau \gamma} +p_{\beta' \alpha'}  
p_{\tau \gamma} p_{\alpha \beta} \big) \A_{\gamma' \alpha}\A_{\tau' \beta}  - 
\big( \ell_{\beta \alpha } \ell_{\tau \gamma } p_{\alpha' \beta'}  +
p_{\beta' \alpha'}\ell_{\alpha' \beta' } \ell_{\tau \gamma }
\big) \A_{\gamma \alpha'}\A_{\tau \beta'}
\\ &
 \quad -  
\left(
 \ell_{\beta \alpha } p_{\tau \gamma}p_{\alpha' \beta'}  +
p_{\beta' \alpha'}\ell_{\alpha' \beta' } p_{\tau \gamma }
\right) \A_{\gamma' \alpha'}\A_{\tau' \beta'} 
\Big] 
\\ &+ 
\big( \ell_{\mu \nu} p_{\tau' \gamma'} \A_{\nu \tau'} \A_{\mu \gamma'} 
 +
p_{\mu \nu} p_{\tau' \gamma'} \A_{\nu' \tau'} \A_{\mu' \gamma'} \big)
\ot
\A_{\gamma \alpha}\A_{\tau \beta} \Big\} .
\end{align*}

Now in the  squared parenthesis, the coefficient $\left(\ell_{\beta \alpha}^2 p_{\tau \gamma} +p_{\beta' \alpha'} p_{\tau \gamma} p_{\alpha \beta} \right) = p_{\tau \gamma}$ is the only one which does not vanish. 
Thus the above reduces to
\begin{align*}
& \sum_{\gamma,\tau}  \big[
-  p_{\tau \gamma}\left(\ell_{\mu \nu} \A_{\nu \tau} \A_{\mu \gamma} +p_{\mu \nu} \A_{\nu' \tau} \A_{\mu' \gamma} \right)  \ot \A_{\gamma' \alpha}\A_{\tau' \beta} 
\\ 
&
\left( \ell_{\mu \nu} p_{\tau' \gamma'} \A_{\nu \tau'} \A_{\mu \gamma'} 
+ p_{\mu \nu} p_{\tau' \gamma'} \A_{\nu' \tau'} \A_{\mu' \gamma'} \right)
 \ot
\A_{\gamma \alpha}\A_{\tau \beta} \big] 
\\ 
= & \sum_{\gamma,\tau} 
\big( -  p_{\tau' \gamma'}\ell_{\mu \nu} \A_{\nu \tau'} \A_{\mu \gamma'} 
 - p_{\tau' \gamma'}p_{\mu \nu} \A_{\nu' \tau'} \A_{\mu' \gamma'} 
+ \ell_{\mu \nu} p_{\tau' \gamma'} \A_{\nu \tau'} \A_{\mu \gamma'} 
 + p_{\mu \nu} p_{\tau' \gamma'} \A_{\nu' \tau'} \A_{\mu' \gamma'} \big)
\\
& \qquad \ot \A_{\gamma \alpha}\A_{\tau \beta}  
=0 ,
\end{align*}
after a replacement  $\gamma \leftrightarrow \gamma'$ and 
$\tau \leftrightarrow \tau'$ (on summed indices) in the first addend. 
\end{proof}

Summarizing, the algebra $\mlp$ of Definition \ref{def:mlp} is a bialgebra with coproduct and counit in \eqref{cop-cou} and it is a transformation bialgebra for $\alp$. In particular we have:
\begin{thm}\label{thm:coaction}
The algebra $\alp$ is a left $\mlp$-comodule algebra with coaction 
\beq\label{coact2}
\delta: \alp \to \mlp \ot \alp ~, \qquad x_\mu \mapsto \sum_\nu \A_{\mu \nu} \ot x_\nu
\eeq
defined on the generators of $\alp$, and extended to the whole of $\alp$ as an  algebra morphism. 
The map $\delta$ extends to a coaction on the differential algebra $(\omlp,d)$ by requiring
\beq\label{dd2}
\delta \circ d = (\id \ot d)\delta \; .
\eeq
Also, the bialgebra $\mlp$ is universal among the bialgebras with these properties. 
\end{thm}

The fact that 
 the bialgebra $\mlp$  is universal (the  initial object) in the category of graded bialgebras coacting on the graded differential algebra $\alp$ follows from the fact that we have determined $\mlp$ by imposing the minimal conditions under which the requirements \eqref{cond1} and \eqref{cond3} are satisfied.  
  
\subsection{The $*$-bialgebra structure of $\mlp$}
We conclude this section by showing that when $\k=\IC$, the bialgebra $\mlp$ can further be endowed with a $*$-structure, and that all maps constructed above in \S\ref{sec:mlp}  are compatible with it.

Let hence $\k=\IC$ and $\bar{\ell}_{\mu \nu}={\ell}_{\mu \nu}$, $\bar{p}_{\mu \nu}= p_{\nu \mu}$ as in \eqref{*lp}, (cf. \S\ref{sec:*lp}).
\begin{lem}\label{lem:*A}
The antilinear map $*$ on $\mlp$, defined on generators as
\beq\label{*A} 
*:\mlp \to  \mlp \; , \quad \A_{\mu \nu} \mapsto (\A_{\mu \nu})^*:=  \A_{\mu \nu}
\eeq
and extended as an anti-algebra map, 
endows $\mlp$ with a well-defined $*$-structure. Furthermore,  the coproduct $\Delta$ and counit $\varepsilon$ of $\mlp$, given in \eqref{cop-cou}, are $*$-morphisms for it.
Thus, $\mlp$ is a $*$-bialgebra.
\end{lem}
\begin{proof}
 We have to show that the commutation relations \eqref{comm-rel-A} defining $\mlp$
 are preserved by the map $*$, that is that
\begin{multline}\label{A*prove}
*\big(\A_{\mu \alpha}\A_{\nu \beta}- \ell_{\mu \nu} \ell_{\beta \alpha } \A_{\nu \beta}\A_{\mu \alpha} -
p_{\mu \nu} \ell_{\beta \alpha } \A_{\nu' \beta}\A_{\mu' \alpha} \\ 
- \ell_{\mu \nu} p_{\beta' \alpha'} \A_{\nu \beta'}\A_{\mu \alpha'}  
-
p_{\mu \nu} p_{\beta' \alpha'} \A_{\nu' \beta'}\A_{\mu' \alpha'} \big) = 0 \, .
\end{multline}
By using $\bar{\ell}_{\mu \nu}={\ell}_{\mu \nu}$ and $\bar{p}_{\mu \nu}= p_{\nu \mu}$, we see  that this is just \eqref{comm-rel-A} for indices  $\mu , \nu$ exchanged:
$$
\A_{\nu \alpha}\A_{\mu \beta}= \ell_{\nu \mu} \ell_{\beta \alpha } \A_{\mu \beta}\A_{\nu \alpha} +
p_{\nu \mu} \ell_{\beta \alpha } \A_{\mu' \beta}\A_{\nu' \alpha} +
\ell_{\nu \mu} p_{\beta' \alpha'} \A_{\mu \beta'}\A_{\nu \alpha'} +
p_{\nu \mu} p_{\beta' \alpha'} \A_{\mu' \beta'}\A_{\nu' \alpha'} \, .
$$
One easily shows, in a similar way, the statement concerning the coalgebra structures.
\end{proof}

It is a direct observation that for $\mlp$ endowed with the above $*$-structure and $\alp$ with  \eqref{*x}, the isomorphism in Lemma~\ref{lem:obs1} is an isomorphism of $*$-algebras.
Finally,
\begin{prop}
The coaction 
$$
\delta: \alp \to \mlp \ot \alp ~, \qquad x_\mu \mapsto \sum_\nu \A_{\mu \nu} \ot x_\nu
$$
given in \eqref{coact} is a $*$-map with respect to the $*$-structures of $\alp$ and  $\mlp$ defined respectively in \eqref{*x} and \eqref{*A}.
\end{prop}

\section{On symmetries of $\slp$}

In the previous section we have constructed a family of matrix bialgebras $\mlp$ coacting on the quantum spaces represented by $\alp$. When all  parameters $p_{\mu \nu}$ vanish and all $\ell_{\mu \nu}$ are equal to $1$ we recover the classical case: the algebra $\alp$ becomes the commutative algebra of polynomials in four coordinates $x_0, \dots,x_3$ and  the bialgebra $\mlp$ reduces to the commutative coordinate bialgebra 
of $4 \times 4$ matrices ${\rm Mat}_4(\k)$. 

We expect it is possible to determine conditions on the parameters $\ell_{\mu \nu}$, $p_{\mu \nu}$  
under which the corresponding bialgebra $\mlp$ admits suitable ideals 
that allow one to define quotient algebras (of $\mlp$) describing     
matrix quantum groups, as for the classical case. 

In particular we would like to determine conditions under which it is possible to define a quantum group of orthogonal matrices acting on $\slp$. 
For this we need to assume at least that conditions \eqref{condition} are satisfied (so that $\slp$ is defined), but we do not know whether these conditions are enough in general.
One needs to show that 
$I:=\langle M^t M - \II ~,~ M  M^t - \II \rangle$ is a well-defined ideal, for $\II$ the identity matrix, i.e. that the diagonal entries of $ M^t   M$ and $M  M^t$ are central in the algebra $\mlp$. Then
one would define $\olp$ to be 
the quotient of $\mlp$ by the ideal $I$.   Since $I$ is a bialgebra ideal,  
$\olp$ would inherit a bialgebra structure and become a Hopf algebra with antipode 
$S(M):= M^t$. Moreover,
the coaction $\delta$ in \eqref{coact} would restrict to a coaction on $\slp$,
being the sphere relation preserved by the coaction:
$$ \delta(\sum_\mu x_\mu x_\mu )= \sum_{\mu,\alpha,\beta} \A_{\mu \alpha} \A_{\mu \beta} \ot x_\alpha x_\beta = \sum_{\alpha,\beta} \delta_{\alpha \beta} \ot
x_\alpha x_\beta = \II \ot \sum_\alpha x_\alpha x_\alpha 
$$
(indeed for this we only need $\A^t \A=\II$). 

As shown in \cite{cdv} this construction can be carried out for the $\theta$-family described in \S \ref{ex:cl}. 
A similar analysis for the more general families will be reported elsewhere.

\end{document}